\documentclass[a4paper,16pt]{article}
\usepackage{amsmath}
\usepackage{amssymb}
\usepackage{amsthm}
\usepackage{geometry}
\usepackage{mathrsfs}
\usepackage{cite}
\usepackage[pdftex]{hyperref}
\usepackage{enumerate}
\def\N{\mathbb{N}}
\def\R{\mathbb{R}}
\def\Z{\mathbb{Z}}
\def\bP{\mathbb{P}}
\def\E{\mathbb{E}}

\def\T{\mathbb{T}}
\def\F{\mathscr{F}}
\def\vep{\varepsilon}
\usepackage[pdftex]{hyperref}
\hypersetup{CJKbookmarks=true}

\theoremstyle{plain}
\newtheorem{theorem}{Theorem}[section]

 \newtheorem{corollary}[theorem]{Corollary}

 \newtheorem{proposition}[theorem]{Proposition}

 \newtheorem{lemma}[theorem]{Lemma}

\newtheorem{remark}{Remark}[section]
 \newtheorem{conj}[theorem]{Conjecture}

  \makeatletter
  \@addtoreset{equation}{section}
  \makeatother

\title{The extremal process of two-speed branching random walk}
\author{Lianghui Luo \footnote{Institut de Math\'ematiques de Toulouse, UMR5219. Universit\'e de Toulouse; CNRS. UPS, F-31062 Toulouse Cedex 9, France. E-mail: lianghui.luo@math.univ-toulouse.fr}}

\begin{document}
\maketitle
\begin{abstract}
 We consider a two-speed branching random walk, which consists of two macroscopic stages with different reproduction laws. We prove that the centered maximum converges in law to a Gumbel variable with a random shift and the extremal process converges in law to a randomly shifted decorated Poisson point process, which can be viewed as a discrete analog for the corresponding results for the two-speed branching Brownian motion, previously established by Bovier and Hartung \cite{bovier}.
\end{abstract}
\bigskip
\noindent\textbf{Keywords:} Branching random walk; extremal process; time-inhomogeneous environment.\\
\noindent\textbf{MSC2020 subject classification:} 60F05,60G70,60J80.
\section{Introduction}
A branching random walk on real line $\R$ can be described as follows. At time 0, one particle is located at 0. At time 1, the particle dies and produces its children according to an independent copy of some point process $\mathcal{L}$. Each of children repeats the behavior of its parent independently. For each particle $u$, we write $|u|$ and $V(u)$ for its generation and position respectively. We call $(V(u):|u|<\infty)$ a branching random walk with reproduction law $\mathcal{L}$.

Under some mild integrability conditions, the asymptotic behavior of the maximal displacement $M_n:=\max_{|u|=n}V(u)$ has been well studied. Let $\kappa(\theta):=\log \E(\sum_{\ell\in \mathcal{L}}e^{\theta \ell})$ be the log-Laplace transform of point process $\mathcal{L}$. As $\kappa(\theta)<\infty$ for some $\theta>0$ and $\kappa(0)>0$, Hammersley\cite{hammersley}, Kingman\cite{kingman} and Biggins\cite{biggins} proved that on the non-extinction set, almost surely,
\[\frac{M_n}{n}\to \inf_{\theta>0}\frac{\kappa(\theta)}{\theta}=:x^*, \quad n\to\infty.\]
Moreover, if there exists some $\theta^*>0$ such that $\kappa'(\theta^*)=x^*$, i.e. $\theta^*\kappa'(\theta^*)=\kappa(\theta^*)$, the logarithmic correction was studied by Hu and Shi\cite{hu}, Addario-Berry and Reed\cite{berry}, and Bramson and Zeitouni\cite{bramson}. They showed that
\[M_n-\kappa'(\theta^*)n+\frac{3}{2\theta^*}\log n\]
is tight conditionally on the non-extinction set.
Finally, A\"id\'ekon\cite{aidekon2} obtained that $M_n-\kappa'(\theta^*)n+\frac{3}{2\theta^*}\log n$ converges in law to a random shift of a Gumbel variable.

The extremal process is a point process formed by $(V(u)-\kappa'(\theta^*)n+\frac{3}{2\theta^*}\log n:|u|=n)$. In the context of branching Brownian motion, Arguin et al.\cite{arguin} and A\"id\'ekon et al.\cite{aidekon1} proved independently the extremal process converges in law to a randomly shifted decorated Poisson point process, and \cite{aidekon1} gives a probabilistic representation of the decoration. For a branching random walk, Madaule\cite{madaule} showed the weak convergence of the extremal process. In Section \ref{sD}, we will introduce these results in detail.

  In recent years,  time-inhomogeneous branching random walk has received a lot of attention, where time-inhomogeneous means that the reproduction law of individuals depend on time. More precisely, fix $M\in \N$, and consider the parameters
\[0=t_0<t_1<\ldots<t_M=1\]
and a collection of point processes $(\mathcal{L}_i:1\leq i\leq M)$. Fix $n\in\N$, at time $0$, here is one individual located at $0$. For $0\leq i\leq M-1$, between time $\lfloor t_in\rfloor$+1 and $\lfloor t_{i+1} n\rfloor$, individuals reproduce their children, which are scattered according to independent copies of the reproduction law $\mathcal{L}_i$, with respect to their parent. We write $V^{(n)}(u)$ for the position of particle $u$ such that $|u|\leq n$ and $M_n^{(n)}:=\max_{|u|=n}V^{(n)}(u)$ for the maximal position of particles at time $n$. When the law of displacement is Gaussian, the first two orders were obtained by Fang and Zeitouni \cite{fang} in the case $M=2$. Under some restriction, Ouimet\cite{Ouimet2} showed the logarithm corrections for $M\geq 2$ by generalizing the approach in \cite{fang}. Under close to optimal conditions, Mallein\cite{bastien} showed the tightness of the centered maximum using spinal decomposition of a time-inhomogeneous branching random walk. In particular, \cite{fang} and \cite{bastien} showed that the log-correction has a phase transition when the law of the second environment is modified.

In this paper, we focus on the case $M=2$ and $t=t_1$, and call $(V^{(n)}(u):|u|\leq n)_{n\geq 0}$ a two-speed branching random walk, namely a time-homogeneous branching random walk. Let $\kappa_i(\theta):=\log \E(\sum_{\ell\in \mathcal{L}_i}e^{\theta \ell})$ be the log-Laplace transform of point process $\mathcal{L}_i$, $i=1,2,$ respectively. Note that $\kappa_i$ is convex and smooth on the interior of $\{\theta\in\R:|\kappa_i(\theta)|<\infty\}.$ Under some integrability conditions, when there exist $\theta_1^*,\theta_2^*>0$ such that
\[\theta_i^*\kappa_i'(\theta_i^*)-\kappa_i(\theta_i^*)=0,\quad i=1,2,\]
and $\theta>0$ such that
\[t[\theta\kappa_1'(\theta)-\kappa_1(\theta)]+(1-t)[\theta\kappa_2'(\theta)-\kappa_2(\theta)]=0,\]
Mallein\cite{bastien} defined three distinct regimes for this model, and proved the tightness of the maximum in each of regimes. The regimes can informally be identified by the sign of $\theta_1^*-\theta_2^*$, being the slow regime if this quantity is negative, mean regime if null or fast regime if positive. However, their characterization does not necessarily relies on the existence of these parameters. More precisely, he proved that $M_n^{(n)}-m_n$ is tight, where
\begin{equation}\label{three}
m_n:=\left\{
         \begin{array}{ll}
           \kappa_1'(\theta_1^*)\lfloor tn\rfloor+\kappa_2'(\theta_2^*)(n-\lfloor tn\rfloor)-(\frac{3}{2\theta_1^*}+\frac{3}{2\theta_2^*})\log n, & \hbox{\text{in the slow regime};} \\
           \kappa_1'(\theta)\lfloor tn\rfloor+\kappa_2'(\theta)(n-\lfloor tn\rfloor)-\frac{3}{2\theta}\log n, & \hbox{\text{in the mean regime};} \\
            \kappa_1'(\theta)\lfloor tn\rfloor+\kappa_2'(\theta)(n-\lfloor tn\rfloor)-\frac{1}{2\theta}\log n, & \hbox{\text{in the fast regime}.}
         \end{array}
       \right.
\end{equation}

In this paper, we consider a two-speed branching random walk $(V^{(n)}(u):|u|\leq n)_{n\geq 0}$ in slow regime or fast regime, and aim to prove the weak convergence of the centered maximum $M^{(n)}_n-m_n$ and extremal process $\mathcal{E}_n:=\sum_{|u|=n}\delta_{V^{(n)}(u)-m_n}$. In the context of two-speed branching Brownian motions, which is a family of branching Brownian motions with inhomogeneous variance, Bovier and Hartung\cite{bovier} proved the weak convergence of the maximum position and the extremal process. Moreover, for a variable speed branching Brownian motion, Bovier and Hartung \cite{bovier1} showed that the limiting extremal process were those that emerged in the two-speed case. For more discussions, see \cite{fang1}, \cite{bovier1}, \cite{pascal}, \cite{bovier2} and \cite{Alban}.

We assume throughout the rest of the paper the two-speed branching random walk survives almost surely for all times and is supercritical, i.e.
\[\bP(\mathcal{L}_i(\R)=0)=0\quad\text{and}\quad \bP(\mathcal{L}_i(\R)=1)<1,\quad i=1,2.\]

\subsection{Fast regime}
We first introduce some basic assumptions in the fast regime. For $\theta>0$ such that $\kappa(\theta)<\infty$, we denote by
\[\kappa'(\theta):=\E\left(\sum_{\ell\in\mathcal{L}}\ell e^{\theta\ell-\kappa(\theta)}\right)\]
if the expectation in the right hand side is well-defined. Note that if $\kappa$ is derivative at point $\theta$, then its derivative is actually $\kappa'(\theta)$, justifying the notation.
 We give four assumptions on $\mathcal{L}$ and $\theta>0$ such that $\kappa(\theta)<\infty$. We assume that
\begin{equation}\label{as1}
\kappa''(\theta):=\E\left(\sum_{\ell\in\mathcal{L}}(\ell-\kappa'(\theta))^2e^{\theta\ell-\kappa(\theta)}\right)\in(0,\infty),
\end{equation}
where if the second derivative of $\kappa$ at point $\theta$ exists and is finite, then its second derivative is $\kappa''(\theta)$, justifying the notation.
 And we also assume that
\begin{equation}\label{as3}
\E\left(\sum_{\ell\in\mathcal{L}}e^{\theta\ell}\log_+\left(\sum_{\ell\in\mathcal{L}}e^{\theta \ell}\right)\right)<\infty,
\end{equation}
where $\log_+(x):=\log (\max\{x,1\})$, for $x\geq 0$.
Finally, we assume that for any $a,b\in\R$,
\begin{equation}\label{as4}
\bP(\ell\in a+b \Z,\forall \ell \in\mathcal{L})<1,
\end{equation}
i.e. $\mathcal{L}$ is non-lattice, where $a+b\Z:=\{a+bz:z\in\Z\}.$

We first present our main result in the fast regime.

\begin{theorem}[Fast regime]\label{th1}
Fix $t\in(0,1)$. Assume that there exists $\theta>0$ such that point processes $\mathcal{L}_1$ and $\mathcal{L}_2$ satisfy assumptions (\ref{as1}), (\ref{as3}), (\ref{as4}) and
\begin{equation}\label{as2}
\kappa_1(\theta)>\theta\kappa'_1(\theta)\quad\text{and}\quad t[\theta\kappa_1'(\theta)-\kappa_1(\theta)]+(1-t)[\theta\kappa_2'(\theta)-\kappa_2(\theta)]=0.
\end{equation} Set
\[m_n:=\frac{\kappa_1(\theta)}{\theta}\lfloor tn\rfloor+\frac{\kappa_2(\theta)}{\theta}(n-\lfloor tn\rfloor)-\frac{1}{2\theta}\log n,\quad n\in\N.\]
Then
\begin{enumerate}[(i)]
  \item there exists some constant $C_{\mathrm{f}}>0$ defined in $(\ref{fast})$ such that for any $y\in\R$,
\begin{equation}\label{re1}
\lim_{n\to\infty}\bP(M_n^{(n)}-m_n\leq y)=\E\left[e^{-C_{\mathrm{f}}W^{(1)}(\theta)e^{-\theta y}}\right],
\end{equation}
where $W^{(1)}(\theta)$ is a positive random variable whose law only depends on $\theta$ and $\mathcal{L}_1$.
  \item The point process
  \begin{equation}
  \mathcal{E}_n:=\sum_{|u|=n}\delta_{V^{(n)}(u)-m_n}
  \end{equation}
  converges in law to a well defined point process $\mathcal{E}$ in the sense of the topology of vague convergence, where
  \begin{equation}\label{limp}
  \mathcal{E}:=\sum_{k,i\geq 1}\delta_{p_k+d_i^{(k)}},
  \end{equation}
  with $p_k$ being the $k$-th atom of a Poisson point process with intensity measure
\[C_{\mathrm{f}}W^{(1)}(\theta)\theta e^{-\theta y}dy,\]
 and $(d_i^{(k)}:i\geq 1)$ being the atoms of independent and identically distributed point processes $\mathcal{D}^{(2,k)}$ (called decorations), whose law only depends on $\mathcal{L}_2$ and $\theta$.
\end{enumerate}
\end{theorem}
\begin{remark}\label{another}
 With the assumption $t[\theta\kappa_1'(\theta)-\kappa_1(\theta)]+(1-t)[\theta\kappa_2'(\theta)-\kappa_2(\theta)]=0$, it is not hard to see that
\begin{equation}\label{mn}
m_n-\left[\kappa_1'(\theta)\lfloor tn\rfloor+\kappa_2'(\theta)(n-\lfloor tn\rfloor)-\frac{1}{2\theta}\log n\right]=\frac{tn-\lfloor tn\rfloor}{1-t}\left(\kappa'_1(\theta)-\frac{\kappa_1(\theta)}{\theta}\right),
\end{equation}
where the right hand side is a bounded sequence.
\end{remark}
Theorem \ref{th1} is an analog of Theorem 1.2 in \cite{bovier} for branching Brownian motion. To be more precise, in the fast regime, (\ref{re1}) tells us that the centered maximum converges in law a Gumbel variable with a random shift, where $W^{(1)}(\theta)$ is the almost surely limit of the additive martingale (defined in forthcoming  (\ref{am})) of a branching random walk $(V_1(u):|u|<\infty)$ with reproduction law $\mathcal{L}_1$. Because $\mathcal{L}_1$ verifies $\theta\kappa_1'(\theta)<\kappa_1(\theta)$ and assumption $(\ref{as3})$, according to \cite{bigginsa}, we have $\bP(W^{(1)}(\theta)=0)=0$, which implies that the limit distribution of the centered maximum is non-degenerate. Moreover, we show that the extremal process converges in law to a randomly shifted decorated Poisson point process, where the decoration is the limit in law of
\begin{equation}\label{eqex}
\sum_{|u|=n}\delta_{V_2(u)-\max_{|u|=n}V_2(u)}
\end{equation}
conditionally on $\{\max_{|u|=n}V_2(u)\geq \kappa_2'(\theta)n\}$, where $(V_2(u):|u|<\infty)$ is a branching random walk with reproduction law $\mathcal{L}_2$. The existence and probabilistic representation of the decoration have been shown in \cite{luo}. We will introduce in more details these quantities in Section \ref{pf1}.

The main idea of the proof is to compute the asymptotic behavior of $\E(e^{-\mathcal{E}_n(f)})$
for all continuous bounded function $f:\R\to [0,\infty)$ whose support has a finite left bound,
where $\mu(f):=\int_{\R}f(x)\mu(\mathrm{d}x)$.
By Markov property at time $\lfloor tn\rfloor$, we can write $\E(e^{-\mathcal{E}_n(f)})=\E(e^{-\mathcal{E}^{(1)}_{\lfloor tn\rfloor}(\phi_{n-\lfloor tn\rfloor}[f])}),$
where $\mathcal{E}_{\lfloor tn\rfloor}^{(1)}:=\sum_{|u|=\lfloor tn\rfloor}\delta_{V_1(u)-\kappa'_1(\theta)\lfloor tn\rfloor}$ and $\phi_{n-\lfloor tn\rfloor}[f](x)=-\log \E[e^{-\mathcal{E}^{(2)}_{n-\lfloor tn\rfloor}(f(\cdot+x))}]$ , where
\[\mathcal{E}^{(2)}_{ n-\lfloor tn\rfloor} :=\sum_{|u|= n-\lfloor tn\rfloor}\delta_{V_2(u)-\kappa'_2(\theta)( n-\lfloor tn\rfloor)+\frac{1}{2\theta}\log n-\frac{tn-\lfloor tn\rfloor}{1-t}\left(\kappa'_1(\theta)-\frac{\kappa_1(\theta)}{\theta}\right)}.\]
 The estimations of $\phi_{n-\lfloor tn\rfloor}[f]$ can be obtained by the results in \cite{luo}, then we are able to compute the asymptotic behavior of $\E(e^{-\mathcal{E}^{(1)}_{\lfloor tn\rfloor}(\phi_{n-\lfloor tn\rfloor}[f])})$.
\subsection{Slow regime}
We first introduce some basic assumptions for the weak convergence of centered maximum of a branching random walk with reproduction law $\mathcal{L}$.
 We assume that there exists $\theta^*>0$ such that
\begin{equation}\label{as6}
\kappa(\theta^*)<\infty,\quad\theta^*\kappa'(\theta^*)=\kappa(\theta^*),
\end{equation}
and
\begin{equation}\label{as7}
\E\left(\sum_{\ell\in\mathcal{L}}(\ell-\kappa'(\theta^*))^2e^{\theta^*\ell-\kappa(\theta^*)}\right)\in(0,\infty).
\end{equation}
And we also assume that
\begin{equation}\label{as8}
\E(X(\log_+X)^2)<\infty\quad\text{and}\quad\E(\tilde{X}\log_+\tilde{X})<\infty,
\end{equation}
where
\[X:=\sum_{\ell\in\mathcal{L}}e^{\theta^*\ell}\quad\text{and}\quad \tilde{X}:=\sum_{\ell\in\mathcal{L}}\max\{-\ell,0\}e^{\theta^* \ell}.\]
\begin{theorem}[Slow regime]\label{th2}
Fix $t\in(0,1)$. Assume that $\mathcal{L}_1,\mathcal{L}_2$ are non-lattice and that there exist $\theta_1^*,\theta_2^*>0$ satisfying $\theta_1^*<\theta_2^*$ and such that $\mathcal{L}_1$ and $\mathcal{L}_2$ satisfy assumptions (\ref{as6}), (\ref{as7}) and (\ref{as8}). Let $m_n:=\kappa_1'(\theta_1^*) \lfloor tn\rfloor+\kappa_2'(\theta_2^*)(n-\lfloor tn\rfloor)-\frac{3}{2\theta_1^*}\log(\lfloor tn\rfloor)-\frac{3}{2\theta_2^*}\log(n-\lfloor tn\rfloor)$ for $n\in\N$. Then
\begin{enumerate}[(i)]
\item there exists some constant $C_{\mathrm{s}}>0$ defined in (\ref{eq6.1}) such that for any $y\in\R$,
\begin{equation}\label{eqth2}
\lim_{n\to\infty}\bP(M_n^{(n)}-m_n\leq y)=\E(e^{-C_{\mathrm{s}}Z^{(1)}e^{-\theta_1^* y}})
\end{equation}
where $Z^{(1)}$ is a positive random variable whose law only depends on $\mathcal{L}_1$.
\item The extremal process
\[\mathcal{E}_n:=\sum_{|u|=n}\delta_{V^{(n)}(u)-m_n}\]
converges in law to a well-defined point process $\mathcal{E}$.
\end{enumerate}
\end{theorem}

Theorem \ref{th2} extends the results in Theorem 1.3 in \cite{bovier} to two-speed branching random walk. It is a direct application of \cite{aidekon2} and \cite{madaule}. In fact, $\bar{Z}$ is the limit of the derivative martingale $Z_n^{(1)}$ (defined in Section \ref{sD}) of a branching random walk $(V_1(u):|u|<\infty)$ with reproduction law $\mathcal{L}_1$. In Section \ref{proofofth2}, we will show that the limit point process $\mathcal{E}$ can be given by
\begin{equation}\label{eqth21}
\mathcal{E}:=\sum_{i,j\geq 1}\delta_{e_i+e_j^{(i)}}
\end{equation}
where $(e_i:i\geq 1)$ and $(e_j^{(i)}:j\geq 1)$ are the atoms of the point process $\mathcal{E}^{(1)}$ and $\mathcal{E}^{(2,i)}$, where $\mathcal{E}^{(1)}$ is the limit in law of the extremal process $\mathcal{E}_n^{(1)}:=\sum_{|u|=n}\delta_{V_1(u)-\kappa_1'(\theta_1^*) n+\frac{3}{2\theta_1^*}\log n}$, and $(\mathcal{E}^{(2,i)})_{i\geq 1}$ are independent of $\mathcal{E}^{(1)}$ and are independent copies of the limit in law of the extremal process $\mathcal{E}^{(2)}_n:=\sum_{|u|=n}\delta_{V_2(u)-\kappa_2'(\theta_2^*)n+\frac{3}{2\theta_2^*}\log n}$. Moreover, by \cite{madaule}, we know that $\mathcal{E}^{(1)}$ is a randomly shifted decorated Poisson point process, which, combined with (\ref{eqth21}), yields that $\mathcal{E}$ is actually a randomly shifted decorated Poisson point process.

\subsection{Mean regime}
In this section, we give a conjecture about the limit law of centered maximum. This case has not been studied in branching Brownian motion setting.
\begin{conj}[Mean regime]
Suppose that $\mathcal{L}_1$ and $\mathcal{L}_2$ are non-lattice. Assume that there exists $\theta_1^*,\theta_2^*>0$ such that $\mathcal{L}_1$ and $\mathcal{L}_2$ satisfy assumptions (\ref{as6}), (\ref{as7}), (\ref{as8}) and $\theta_1^*=\theta_2^*$. Let $\theta:=\theta_1^*$ and $m_n:=\kappa_1'(\theta) \lfloor tn\rfloor+\kappa_2'(\theta) (n-\lfloor tn\rfloor)-\frac{3}{2\theta}\log n$. Then
\begin{enumerate}[(i)]
\item there exists some constant $C_{\mathrm{m}}>0$ such that for any $y\in\R$, we have
\[\lim_{n\to\infty}\bP(M_n^{(n)}-m_n\leq y)=\E(e^{-C_{\mathrm{m}}Z^{(1)}e^{-\theta y}}),\]
where $Z^{(1)}$ is the same random variable as in Theorem \ref{th2}.
\item the extremal process
\[\mathcal{E}_n:=\sum_{|u|=n}\delta_{V^{(n)}(u)-m_n}\]
converges in law to a randomly shifted decorated Poisson point process $\mathcal{E}$ with the random intensity measure $C_{\mathrm{m}}Z^{(1)}\theta e^{-\theta y}dy$ and the decoration $\mathcal{D}^{(2)}$, where $\mathcal{D}^{(2)}$ is the limit in law of $\sum_{|u|=n}\delta_{V_2(u)-\max_{|u|=n}V_2(u)}$ conditionally on $\{\max_{|u|=n}V_2(u)\geq \kappa_2'(\theta) n\}$.
\end{enumerate}
\end{conj}

The rest of the paper is organized as follows. In Section \ref{s2}, we give a construction of two-speed branching random walk. In the rest of Section \ref{s2}, we focus on a branching random walk, and introduce the additive martingale and the spinal decomposition theorem for a time-homogeneous branching random walk. Moreover, we use spinal decomposition theorem to prove a $L^1$ convergence theorem associated with additive martingales. In the final of Section \ref{s2}, some well-known results about derivative martingales and extremal processes are introduced, and we give some properties about extremal processes. In Section \ref{pf1} and Section \ref{proofofth2}, we give proofs of Theorem \ref{th1} and Theorem \ref{th2} respectively.

In the proof of this paper, we denote by $C$ positive constant that may change from line to line.
\section{Preliminaries}\label{s2}
For convenience, we write $t_n$ for $\lfloor tn\rfloor$. In this section, we explain how to construct a two-speed branching random walk. Then we go back to the setting of a time-homogeneous branching random walk and introduce the additive martingale and spinal decomposition theorem, which help us to prove a limit theorem associated with additive martingales. Next, we introduce the derivative martingale and some classical results about the limit of centered maximum and extremal process, and prove some properties about the limit of the extremal process.
\subsection{Construction of the two-speed branching random walk}\label{uhn}
 Define the set of finite sequence of positive integers
\[\mathbb{T}:=\bigcup_{n\geq 0}\N^{n},\]
with the usual convention $\N^0:=\{\varnothing\}$.

Let $(\mathcal{L}_i^{(u)}:u\in\T)$ be independent copies of point process $\mathcal{L}_i$, $i=1,2$ respectively. We assume that $(\mathcal{L}_1^{(u)}:u\in\T)$ and $(\mathcal{L}_2^{(u)}:u\in\T)$ are independent. As there exists some $\theta_i>0$ such that $\kappa_i(\theta_i)<\infty$, then for any $y\in\R$, $\mathcal{L}_i([y,\infty))<\infty$ almost surely, therefore we can rank the atoms of $L_i$ in a non-increasing fashion. For $i\in\{1,2\}$, $u\in\T$, we write $\mathcal{L}_i^{(u)}$ as $(\ell_{i,j}^{(u)}:j\geq 1)$, where $(\ell_{i,j}^{(u)}:j\geq 1)$ is a non-increasing sequence generated by all atoms in $\mathcal{L}_i^{(u)}$ with the conventions that $\ell_{i,j}^{(u)}:=-\infty$ if $j>\mathcal{L}_{i}^{(u)}(\R)$ and that $\lim_{j\to\infty} \ell_{i,j}^{(u)}=-\infty$.
For $u\in\T$, we write $|u|$ for its generation, i.e. $u\in\N^{|u|}$. For $0\leq j\leq |u|$, we denote by $u_j$ the $j$-th ancestor of particle $u$, i.e. $u_j$ is the projection of $u$ on $\N^{j}$ . With the help of notation above, we can construct a two-speed branching random walk $(V^{(n)}(u):|u|\leq n)_{n\geq 1}$ as follows. Fix $t\in[0,1]$, for any $n\in\N$, let $V^{(n)}(\varnothing):=0$ and for $|u|\leq n$ and $u\neq \varnothing$,
\[V^{(n)}(u):=\left\{
               \begin{array}{ll}
                 \sum_{j=1}^{|u|}\ell_{1,j}^{(u_{j-1})}, & \hbox{$1\leq |u|\leq t_n$;} \\
                 \sum_{j=1}^{t_n}\ell_{1,j}^{(u_{j-1})}+\sum_{j=t_n+1}^{|u|}\ell_{2,j}^{(u_{j-1})}, & \hbox{$t_n+1\leq |u|\leq n$,}
               \end{array}
             \right.
\]
 where we say that particle $u$ is in death if $V^{(n)}(u)=-\infty$. We denote by $\{|u|=n\}$ the set of all particles alive at time $n$. Remark that if $t\in\{0,1\}$, a two-speed branching random walk can be seen as a time-homogeneous branching random walk. We thus assume that $t\in(0,1)$.
\subsection{Spinal decomposition: a change of measure}\label{sd}
In this section, we only consider a time-homogeneous branching random walk $(V(u):u\in\T)$ with reproduction law $\mathcal{L}$. Lyons\cite{lyons} introduced the spinal decomposition theorem for the branching random walk as an alternative description of the law of the branching random biased by its additive martingale. This method is an extension of the spinal decomposition obtained by Lyons et al.\cite{lyons1} for Galton-Watson processes and Chauvin and Rouault\cite{chauvin} for branching Brownian motion.

Let $(\mathscr{F}_n)_{n\geq 0}$ be the natural filtration of the branching random walk, i.e.
 \[\mathscr{F}_n:=\sigma (V(u):|u|\leq n),\]
 with $\mathscr{F}_\infty:=\sigma(V(u):u\in\mathbb{T})$.
Fix $\theta>0$ such that $\kappa(\theta):=\E(\sum_{|u|=1}e^{\theta V(u)})<\infty$, we denote by
\begin{equation}\label{am}
W_n(\theta):=\sum_{|u|=n}e^{\theta V(u)-n\kappa(\theta)}
\end{equation}
the additive martingale with parameter $\theta$.
 By the branching property, $(W_n(\theta))_{n\geq 1}$ is a non-negative martingale with respect to $(\F_n)_{n\geq 1}$, which implies that there exists a non-negative random variable $W(\theta):=\lim_{n\to\infty} W_n(\theta)$ a.s.

For $x\in\R$, we denote by $\bP_x$ the law of branching random walk starting from 0 and $\E_x$ the corresponding expectation. We write $\bP$ and $\E$ instead of $\bP_0$ and $\E_0$ for brevity. On the one hand, using Kolmogorov's extension theorem, there exists a probability $\tilde{\bP}_x$ on $\mathscr{F}_\infty$ satisfying for any $n\geq 0$, $A\in \mathscr{F}_n$,
\[\tilde{\bP}_x(A)=e^{-\theta x}\E_x(W_n(\theta)1_{A}).\]
We write $\tilde{\E}_x$ for the expectation with respect to $\tilde{\bP}_x$.

On the other hand, we define a point process $\hat{\mathcal{L}}$ is such that the law of $\hat{\mathcal{L}}$ is equal to the law of $\mathcal{L}$ biased by $\sum_{\ell\in \mathcal{L}}e^{\theta \ell-\kappa(\theta)}$. A branching random walk $(V(u):u\in\mathbb{T})$ with a spine $(\xi_n:n\geq 0)$ can be described as follows. At time 0, one particle $\xi_0:=\varnothing$ locates at position $V(\xi_0)=x$. At each time $n\geq1$, all particles die, while giving birth independently to sets of new particles. The displacements (with respect to their parent) of children of normal particle $u$ are distributed as $(V(i):i\geq 1)$. The displacements (with respect to their parent) of children of spine particle $\xi_{n-1}$ are distributed as $\hat{\mathcal{L}}$; the particle $\xi_n$ is chosen among the children $v$ of $\xi_{n-1}$ with probability proportional to $e^{\theta V(v)}$. Let us denote by $\bar{\bP}_x$ the law of the branching random walk with a spine and $\bar{\E}_x$ the corresponding expectation. The spinal decomposition theorem corresponds in the identification of the law of the size-biased branching random walk and the law of the branching random walk with a spine.
 \begin{theorem}[Lyons\cite{lyons}]
 For any $x\in\R$ and $n\geq 0$, the law of $(V(u):|u|\leq n)$ is identical under $\tilde{\bP}_x$ and $\bar{\bP}_x$. Moreover, for any $|u|=n$,
\begin{equation}\label{spine}
\bar{\bP}_x(u=\xi_n|\mathscr{F}_n)=\frac{e^{\theta V(u)-n\psi(\theta)}}{W_n(\theta)}\quad \text{a.s.}
\end{equation}
 \end{theorem}
Finally, we introduce the many-to-one formula, which can be seen as an direct corollary of the spinal decomposition theorem.
\begin{corollary}[Theorem 1.1 in \cite{shi}]\label{many}
Assume that there exists $\theta>0$ such that $\kappa(\theta)<\infty$.
Then for any $x\in\R$ and $n\in\N$ and any measurable function $g$: $\R^n\to[0,\infty)$, we have
\[\E_x(\sum_{|u|=n}g(V(u_1),\ldots,V(u_n)))=e^{\theta x}\E_x(e^{-\theta S_n+n\kappa(\theta)}g(S_1,\ldots,S_n)),\]
where under $\bP_x$, $(S_k)_{k\geq 0}$ is a random walk starting from $x$ and is such that for any measurable function $f:\R\to [0,\infty)$,
\[e^{-\theta x}\E_x(\sum_{|u|=1}f(V(u))e^{\theta V(u)-\kappa(\theta)})=\E_x(f(S_1)).\]
\end{corollary}
Moreover, as $\theta>0$ such that $\kappa'(\theta),\kappa''(\theta)<\infty$, we know that $\E(S_1)=\kappa'(\theta)$ and $\E[(S_1-\E(S_1))^2]=\kappa''(\theta)$.
\subsection{Central limit theorem associated with additive martingales}\label{adm}
We denote by $C_b(\R)$ the set of all continuous and bounded function $f:\R\to\R.$ Given a time-homogeneous branching random walk $(V(u):u\in\T)$ with reproduction law $\mathcal{L}$, we aim to prove the following theorem.
\begin{theorem}\label{prop1}
Let $\theta>0$ such that assumptions (\ref{as1}), (\ref{as3}) hold and such that $\theta\kappa'(\theta)<\kappa(\theta)$. Then for any $f\in C_b(\R)$, we have
\begin{equation}\label{L1cov}
\lim_{n\to\infty}\sum_{|u|=n}e^{\theta V(u)-n\kappa(\theta)}f\left(\frac{V(u)-n\kappa'(\theta)}{\sqrt{n}}\right)= W(\theta)\E(f(N))\quad \text{in } L^1,
\end{equation}
where $N\sim \mathcal{N}(0,\kappa''(\theta))$.
\end{theorem}
\begin{remark}
For our purpose, we only need to prove that the convergence in Theorem \ref{prop1}
holds in probability. Under the additional assumption that $\E(W_1(\theta)^{1+\gamma})<\infty$ for some $\gamma>0$, this result has been proved in \cite{bigginsu}, \cite{pain} and \cite{Ik}. We remove this integrability condition by a truncation argument.
\end{remark}
We first prove a series of lemmas, which are helpful for us to prove Theorem \ref{prop1} in the end of this section. Fix $f\in C_b(\R)$,
we define
\[\bar{W}_n(\theta):=\sum_{|u|=n}e^{\theta V(u)-n\kappa(\theta)}f\left(\frac{V(u)-n\kappa'(\theta)}{\sqrt{n}}\right).\]
Recall that $\F_k=\sigma(V(u):|u|\leq k)$.
\begin{lemma}\label{lem1}
Under the assumptions of Theorem \ref{prop1}, for any $f\in C_b(\R)$, we have
\[
\lim_{k\to\infty}\limsup_{n\to\infty}\E\left(\left|\E(\bar{W}_n(\theta)|\mathscr{F}_k)-W(\theta)\E f(N)\right|\right)=0.
\]
\end{lemma}
\begin{proof}
By triangular inequality, we have
\begin{equation}\label{ineq1}
\begin{split}
&\E\left(\left|\E(\bar{W}_n(\theta)|\mathscr{F}_k)-W(\theta)\E f(N)\right|\right)\\
\leq &\E\left(\left|W_k(\theta)\E f(N)-\E(\bar{W}_n(\theta)|\mathscr{F}_k)\right|\right)+\E(|W_k(\theta)-W(\theta)|)\E f(N).
\end{split}
\end{equation}
Using the many-to-one formula,
\[\E(\bar{W}_n(\theta)|\mathscr{F}_k)=\sum_{|u|=k}e^{\theta V(u)-k\kappa(\theta)}\E_{V(u)}\left[ f\left(\frac{S_{n-k}-n\kappa'(\theta)}{\sqrt{n}}\right)\right],\]
which implies that for fixed $k\in\N$,
\begin{equation}\label{eqa}
\begin{split}
|W_k(\theta)\E f(N)-\E&(\bar{W}_n(\theta)|\mathscr{F}_k)|\\
&\leq \sum_{|u|=k}e^{\theta V(u)-k\kappa(\theta)}\left|\E f(N)-\E_{V(u)}\left[ f\left(\frac{S_{n-k}-n\kappa'(\theta)}{\sqrt{n}}\right)\right]\right|
\end{split}
\end{equation}
Because $W_n(\theta)$ is a martingale, we know that $\E(W_n(\theta))=\E(W_0(\theta))=1$, which yields that $W_k(\theta)<\infty$ almost surely. For $|u|=k$, by central limit theorem, we know that almost surely
\[\lim_{n\to\infty}\left|\E f(N)-\E_{V(u)}\left[ f\left(\frac{S_{n-k}-n\kappa'(\theta)}{\sqrt{n}}\right)\right]\right|=0.\]
Thus, by dominated convergence, we conclude that the right hand side of (\ref{eqa}) decays to 0 almost surely as $n\to\infty$, then
\begin{equation}\label{eq1}
\lim_{n\to\infty}\E|W_k(\theta)\E f(N)-\E(\bar{W}_n(\theta)|\F_k)|=0.
\end{equation}
Combined with (\ref{ineq1}), we observe that for any $k\in\N$,
\[\limsup_{n\to\infty} \E\left(\left|\E(\bar{W}_n(\theta)|\F_k)-W(\theta)\E f(N)\right|\right)\leq \E(|W_k(\theta)-W(\theta)|)\E f(N).\]

By Theorem A in \cite{bigginsa}, $W_k(\theta)$ converges to $W(\theta)$ in $L^1$. Combined with the fact that $f$ is bounded, we have
\begin{equation}\label{eq2}
\lim_{k\to\infty}\E(|W_k(\theta)-W(\theta)|)\E f(N)=0,
\end{equation}
which completes the proof.
\end{proof}
As $\theta\kappa'(\theta)<\kappa(\theta)$, we can fix $a,L>0$ such that $a+\theta L+\theta\kappa'(\theta)-\kappa(\theta)<0$. For $u=(u^{(1)},\ldots,u^{(|u|)})$, $v=(v^{(1)},\ldots,v^{(|v|)})\in\T$, let $uv:=(u^{(1)},\ldots,u^{(|u|)},v^{(1)},\ldots,v^{|v|})$ be the concatenation operate with the convention $u\varnothing=\varnothing u=u$.  We introduce some truncated set as follows. For any $A>0$ and $n\in\N$, we define
\[E_1(n,A):=\left\{u\in\T: \sum_{i=1}^{\infty}e^{\theta (V(u_ki)-V(u_k))}<Ae^{a(k+1)},\forall 0\leq k\leq n-1\right\},\]
where $E_1(n,A)$ controls the number and positions of siblings of particles with respect to their parents.
And we denote by
 \[E_2(n,A):=\{u\in\T: V(u_k)-k\kappa'(\theta)\leq Lk+A,\forall 1\leq k\leq n\}\]
 a collection of particle $u$ whose trajectory $\{V(u_k):1\leq k\leq n\}$ stays below the line $\{(L+\kappa'(\theta))k+A:1\leq k\leq n\}$.
We construct an auxiliary process
\[\bar{W}_{n,A}(\theta):=\sum_{|u|=n}e^{\theta V(u)-n\kappa(\theta)}f\left(\frac{V(u)-n\kappa'(\theta)}{\sqrt{n}}\right)1_{\{u\in E_1(n,A)\cap E_2(n,A)\}}.\]
The following lemma tells us that $\bar{W}_{n,A}(\theta)$ is a good approximation of $\bar{W}_n(\theta)$.
\begin{lemma}\label{lem2}
Fix $\theta>0$ such that assumption (\ref{as1}) holds and $\theta\kappa'(\theta)<\kappa(\theta)$. Then
\begin{equation}\label{eq3}
\lim_{A\to\infty}\limsup_{n\to\infty}\sup_{k\in\N}\E(|\E(\bar{W}_n(\theta)|\mathscr{F}_k)-\E(\bar{W}_{n,A}(\theta)|\mathscr{F}_k)|)=0
\end{equation}
and
\begin{equation}\label{eq4}
\lim_{A\to\infty}\limsup_{n\to\infty}\E(|\bar{W}_n(\theta)-\bar{W}_{n,A}(\theta)|)=0.
\end{equation}
\end{lemma}
\begin{proof}
By Jensen inequality, for any $k\in\N$,
\[\E(|\E(\bar{W}_n(\theta)|\mathscr{F}_k)-\E(\bar{W}_{n,A}(\theta)|\mathscr{F}_k)|)\leq \E(|\bar{W}_n(\theta)-\bar{W}_{n,A}(\theta)|).\]
Thus, (\ref{eq4}) implies (\ref{eq3}). We denote by $\|f\|_\infty$ the essential supremum of $|f|$.
Note that
\begin{equation}\label{eq5}
\begin{split}
&\E(|\bar{W}_n(\theta)-\bar{W}_{n,A}(\theta)|)\\
\leq &\|f\|_\infty\E\left(\sum_{|u|=n}e^{\theta V(u)-n\kappa(\theta)}1_{\{u\notin E_1(n,A)\}\cup\{u\notin E_2(n,A)\}}\right)\\
\leq &\|f\|_\infty\times\left[\E\left(\sum_{|u|=n}e^{\theta V(u)-n\kappa(\theta)}1_{\{u\notin E_1(n,A)\}}\right)+\E\left(\sum_{|u|=n}e^{\theta V(u)-n\kappa(\theta)}1_{\{u\notin E_2(n,A)\}}\right)\right].
\end{split}
\end{equation}

On the one hand, by many-to-one formula, we have
\begin{equation}\label{eq3.7}
\begin{split}
\E\left(\sum_{|u|=n}e^{\theta V(u)-n\kappa(\theta)}1_{\{u\notin E_2(n,A)\}}\right)
=&\bP(\exists 1\leq k\leq n, S_k-k\kappa'(\theta)>Lk+A)\\
\leq &\sum_{k=1}^\infty\bP(S_k>(L+\kappa'(\theta))k+A).
\end{split}
\end{equation}
Recall that $\E(S_1)=\kappa'(\theta)$, because $\E(S_1^2)<\infty$ and $L>0$, by Hsu-Robbins theorem (see \cite{erdos}),  we obtain
 \[\sum_{k=1}^\infty\bP(S_k>(L+\kappa'(\theta))k)<\infty.\]
Combined with (\ref{eq3.7}), by monotone convergence theorem, we conclude that
\begin{equation}\label{eq3.8}
\lim_{A\to\infty}\sup_{n\in\N}\E\left(\sum_{|u|=n}e^{\theta V(u)-n\kappa(\theta)}1_{\{u\notin E_2(n,A)\}}\right)=0.
\end{equation}

On the other hand, according to spinal decomposition theorem,
\begin{equation}\label{eq3.5}
\begin{split}
&\E\left(\sum_{|u|=n}e^{\theta V(u)-n\kappa(\theta)}1_{\{u\notin E_1(n,A)\}}\right)\\
=&\bar{\E}(1_{\{\xi_n=u\}}1_{\{u\notin E_1(n,A)\}})\\
=&\bar{\bP}\left(\exists 0\leq k\leq n-1, \sum_{i=1}^\infty e^{\theta (V(\xi_ki)-V(\xi_k))}\geq Ae^{a(k+1)}\right).
\end{split}
\end{equation}
Recall that $(V(\xi_ki)-V(\xi_k):i\geq 1)_{k\geq 0}$ are i.i.d under $\bar{\bP}$, we have
\begin{equation}\label{eq3.6}
\begin{split}
&\bar{\bP}\left(\exists 0\leq k\leq n-1, \sum_{i=1}^\infty e^{\theta (V(\xi_ki)-V(\xi_k))}\geq Ae^{a(k+1)}\right)\\
\leq &\sum_{k=1}^\infty\bar{\bP}\left(\log\sum_{|u|=1}e^{\theta V(u)}>ak+\log A\right).
\end{split}
\end{equation}
From assumption $\E(\sum_{|u|=1}e^{\theta V(u)}\log\sum_{|u|=1}e^{\theta V(u)} )<\infty$ and spinal decomposition theorem, we know that $\bar{\E}(\log\sum_{|u|=1}e^{\theta V(u)})<\infty$, combined with $a>0$, which implies that
\[\sum_{k=1}^\infty \bar{\bP}\left(\log\sum_{|u|=1}e^{\theta V(u)}>ak\right)<\infty.\]
Combined with (\ref{eq3.5}) and (\ref{eq3.6}), by monotone convergence theorem, we have
\[\lim_{A\to\infty}\sup_{n\in\N}\E\left(\sum_{|u|=n}e^{\theta V(u)-n\kappa(\theta)}1_{\{u\notin E_1(n,A)\}}\right)=0,\]
which, combined with (\ref{eq5}) and (\ref{eq3.8}), implies (\ref{eq4}).
\end{proof}
\begin{lemma}\label{second}
Fix $\theta>0$ such that $\kappa(\theta)<\infty$ and $\kappa'(\theta)<\infty$. For any $a,A,L>0$, $n\in\N$, $0\leq k\leq n-1$, we have
\[
\begin{split}
&\E_{x}\left[\left(\sum_{|v|=n-k}e^{\theta V(v)-(n-k)\kappa(\theta)}1_{\{v\in E_1(n-k,Ae^{ak})\cap E_2(n-k,A+Lk+\kappa'(\theta)k)\}}\right)^2\right]\\
\leq &e^{\theta x}(A+1)e^{\theta A+a}e^{(a+\theta L+\theta \kappa'(\theta))k}\sum_{i=0}^{n-k}e^{( a+\theta L+\theta\kappa'(\theta)-\kappa(\theta))i}.
\end{split}
\]

\end{lemma}
\begin{proof}
Let $E(n,k,A):=E_1(n-k,Ae^{ak})\cap E_2(n-k,A+Lk+\kappa'(\theta)k)$. For fixed $x\in\R$ and $0\leq k\leq n-1$, by spinal decomposition theorem,
\begin{align}
 &\E_{x}\left[\left(\sum_{|v|=n-k}e^{\theta V(v)-(n-k)\kappa(\theta)}1_{\{v\in E(n,k,A)\}}\right)^2\right]\nonumber\\
=&e^{\theta x}\bar{\E}_x\left(1_{\{\xi_{n-k}\in E(n,k,A)\}}\sum_{|v|=n-k}e^{\theta V(v)-(n-k)\kappa(\theta)}1_{\{v\in E(n,k,A)\}}\right)\nonumber\\
\leq &e^{\theta x}\bar{\E}_x\left(1_{\{\xi_{n-k}\in E(n,k,A)\}}\sum_{|v|=n-k}e^{\theta V(v)-(n-k)\kappa(\theta)}\right).\label{eq9}
\end{align}
For $u,v\in \T$, we write $u\prec v$ or $v\succ u$, if $u$ is the ancestor of $v$. By decomposing $(V(u):|u|\leq n-k)$ along the spine, it is not hard to see
\[
\begin{split}
&\sum_{|v|=n-k}e^{\theta V(v)-(n-k)\kappa(\theta)}\\
=&e^{\theta V(\xi_{n-k})-(n-k)\kappa(\theta)}\\
&\quad+\sum_{i=0}^{n-k-1}\sum_{j\geq 1\atop j\neq \xi^{(i+1)}_{n-k}}\sum_{|v|=n-k\atop v\succ\xi_ij}e^{\theta (V(\xi_i)+V(\xi_ij)-V(\xi_i)+V(v)-V(\xi_ij))-(n-k)\kappa(\theta)},
\end{split}
\]
which implies that
\begin{equation}\label{eq8}
\begin{split}
&\bar{\E}_x\left(1_{\{\xi_{n-k}\in E(n,k,A)\}}\sum_{|v|=n-k}e^{\theta V(v)-(n-k)\kappa(\theta)}\right)\\
\leq &\bar{\E}_x(e^{\theta V(\xi_{n-k})-(n-k)\kappa(\theta)}1_{\{\xi_{n-k}\in B(n-k,k)\}})\\
&+\bar{\E}_x\left(\sum_{i=0}^{n-k-1}\sum_{j\geq 1\atop j\neq\xi^{(i+1)}_{n-k}}\sum_{|v|=n-k\atop v\succ\xi_ij}e^{\theta (V(\xi_i)+V(\xi_ij)-V(\xi_i)+V(v)-V(\xi_ij))-(n-k)\kappa(\theta)}1_{\{\xi_i\in B(i,k)\}}\right),
\end{split}
\end{equation}
where $B(i,k):=\{|u|=i:V(u)-(k+i)\kappa'(\theta)\leq L(k+i)+A, \sum_{j\geq 1}e^{\theta (V(uj)-V(u))}\leq Ae^{a(k+i+1)}\}.$
By independence, for $0\leq i\leq n-k-1$,
\[
\begin{split}
&\bar{\E}_x\left(\sum_{j\geq 1\atop j\neq \xi^{(i+1)}_{n-k}}\sum_{|v|=n-k\atop v\succ\xi_ij}e^{\theta (V(\xi_i)+V(\xi_ij)-V(\xi_i)+V(v)-V(\xi_ij))-(n-k)\kappa(\theta)}1_{\{\xi_i\in B(i,k)\}}\right)\\
=&\bar{\E}_x\left(e^{\theta V(\xi_i)-i\kappa(\theta)}\sum_{j\geq 1\atop j\neq\xi^{(i+1)}_{n-k}}e^{\theta(V(\xi_ij)-V(\xi_i))}1_{\{\xi_i\in B(i,k)\}}\right),
\end{split}
\]
which, combined with (\ref{eq9}) and (\ref{eq8}), implies that
\[
\begin{split}
&\E_{x}\left(\sum_{|v|=n-k}e^{\theta V(v)-(n-k)\kappa(\theta)}1_{\{v\in E(n,k,A)\}}\right)^2\\
\leq&  e^{\theta x}\bar{\E}_x(e^{\theta V(\xi_{n-k})-(n-k)\kappa(\theta)}1_{\{\xi_{n-k}\in B(n-k,k)\}})\\
&\quad+e^{\theta x}\sum_{i=0}^{n-k-1}\bar{\E}_x\left(e^{\theta V(\xi_i)-i\kappa(\theta)}\sum_{j\geq 1}e^{\theta(V(\xi_ij)-V(\xi_i))}1_{\{\xi_i\in B(i,k)\}}\right)\\
\leq &e^{\theta x}(A+1)e^{\theta A+a}e^{(a+\theta L+\theta \kappa'(\theta))k}\sum_{i=0}^{n-k}e^{( a+\theta L+\theta\kappa'(\theta)-\kappa(\theta))i},
\end{split}
\]
which completes the proof.
\end{proof}

\begin{lemma}\label{lem3}
Fix $\theta>0$ such that $\kappa(\theta)<\infty$ and $\theta\kappa'(\theta)<\kappa(\theta)$. Then for any $A>0$,
\[\lim_{k\to\infty}\limsup_{n\to\infty}\E((\bar{W}_{n,A}(\theta)-\E(\bar{W}_{n,A}(\theta)|\mathscr{F}_k))^2)=0.\]
\end{lemma}
\begin{proof}
By Lemma $\ref{second}$ and $f\in C_b$, we know that
\[\E(\bar{W}_{n,A}^2(\theta))<\infty,\]
which yields that
\begin{equation}\label{eq10}
\E((\bar{W}_{n,A}(\theta)-\E(\bar{W}_{n,A}(\theta)|\mathscr{F}_k))^2)=\E(\E(\bar{W}_{n,A}^2(\theta)|\mathscr{F}_k)-(\E(\bar{W}_{n,A}(\theta)|\mathscr{F}_k))^2).
\end{equation}
Therefore, we shall compute $\E(\bar{W}_{n,A}^2(\theta)|\mathscr{F}_k)$ and $\E(\bar{W}_{n,A}(\theta)|\mathscr{F}_k)^2$ respectively. Recall that $E(n,k,A)=E_1(n-k,Ae^{ak})\cap E_2(n-k,A+Lk+\kappa'(\theta)k)$.
 By Markov property, we have
\begin{align}
\E&(\bar{W}_{n,A}^2(\theta)|\mathscr{F}_k)\nonumber\\
&=\sum_{|u|=|v|=k,u\neq v}e^{\theta V(u)+\theta V(v)-2k\kappa(\theta)}1_{\{u,v\in E_1(k,A)\cap E_2(k,A)\}}H_{n,k,A}(V(u))H_{n,k,A}(V(v))\nonumber\\
&+\sum_{|u|=k}e^{2\theta V(u)-2k\kappa(\theta)}1_{\{u\in E_1(k,A)\cap E_2(k,A)\}}G_{n,k,A}(V(u))\label{eq6}
\end{align}
where for $x\in\R$
\[H_{n,k,A}(x):=e^{-\theta x}\E_{x}\left(\sum_{|v|=n-k}e^{\theta V(v)-(n-k)\kappa(\theta)}f\left(\frac{V(v)-n\kappa'(\theta)}{\sqrt{n}}\right)1_{\{v\in E(n,k,A)\}}\right)\]
and
\[G_{n,k,A}(x):=e^{-2\theta x}\E_{x}\left[\left(\sum_{|v|=n-k}e^{\theta V(v)-(n-k)\kappa(\theta)}f\left(\frac{V(v)-n\kappa'(\theta)}{\sqrt{n}}\right)1_{\{v\in E(n,k,A)\}}\right)^2\right].\]
On the other hand, by Markov property again, we obtain
\begin{align}
&\E(\bar{W}_{n,A}(\theta)|\mathscr{F}_k)^2\nonumber\\
=&\left(\sum_{|u|=k}e^{\theta V(u)-k\kappa(\theta)}1_{\{u\in E_1(k,A)\cap E_2(k,A)\}}H_{n,k,A}(V(u))\right)^2\nonumber\\
=&\sum_{|u|=|v|=k,u\neq v}e^{\theta V(u)+\theta V(v)-2k\kappa(\theta)}1_{\{u,v\in E_1(k,A)\cap E_2(k,A)\}}H_{n,k,A}(V(u))H_{n,k,A}(V(v))\nonumber\\
&+\sum_{|u|=k}e^{2\theta V(u)-2k\kappa(\theta)}1_{\{u\in E_1(k,A)\cap E_2(k,A)\}}(H_{n,k,A}(V(u)))^2 .\label{eq7}
\end{align}
Combining (\ref{eq6}) with (\ref{eq7}), we have
\begin{equation}\label{eq11}
\begin{split}
&\E(\bar{W}_{n,A}(\theta)^2|\mathscr{F}_k)-\E(\bar{W}_{n,A}(\theta)|\mathscr{F}_k)^2\\
\leq&\sum_{|u|=k}e^{2\theta V(u)-2k\kappa(\theta)}1_{\{u\in E_1(k,A)\cap E_2(k,A)\}}G_{n,k,A}(V(u)).
\end{split}
\end{equation}
Recall that $A>0$, $a>0$, $L>0$ such that $a+\theta L+\theta\kappa'(\theta)-\kappa(\theta)<0$, by Lemma \ref{second}, we conclude that
\[
\begin{split}
G_{n,k,A}(x)\leq &\|f\|_\infty^2 (A+1)e^{\theta A-\theta x+a}e^{(a+\theta L+\theta \kappa'(\theta))k}\sum_{i=0}^{n-k}e^{( a+\theta L+\theta\kappa'(\theta)-\kappa(\theta))i}\\
\leq &\frac{\|f\|_\infty^2(A+1)e^{\theta A-\theta x+a}}{1-e^{a+\theta L+\theta\kappa'(\theta)-\kappa(\theta)}}e^{(a+\theta L+\theta \kappa'(\theta))k}
= C(A,f)e^{-\theta x} e^{(a+\theta L+\theta \kappa'(\theta))k}
\end{split}
\]
where $C(A,f):=\frac{\|f\|_\infty^2Ae^{\theta A+a}}{1-e^{a+\theta L+\theta\kappa'(\theta)-\kappa(\theta)}}$.
Thus, we have
\begin{equation}\label{eq12}
\begin{split}
&\limsup_{n\to\infty}\E\left(\sum_{|u|=k}e^{2\theta V(u)-2k\kappa(\theta)}1_{\{u\in E_1(k,A)\cap E_2(k,A)\}}G_{n,k,A}(V(u))\right)\\
\leq &C(A,f)e^{(a+\theta L+\theta \kappa'(\theta))k}\E\left(\sum_{|u|=k}e^{2\theta V(u)-2k\kappa(\theta)}1_{\{u\in E_1(k,A)\cap E_2(k,A)\}}e^{-\theta V(u)}\right)\\
\leq &C(A,f)e^{(a+\theta L+\theta \kappa'(\theta))k}\E\left(\sum_{|u|=k}e^{\theta V(u)-2k\kappa(\theta)}\right)=C(A,f)e^{(a+\theta L+\theta\kappa'(\theta)-\kappa(\theta))k}
\end{split}
\end{equation}
where the last equality follows from the fact that $W_n(\theta)$ is a martingale with $\E(W_0(\theta))=1$.
Combining (\ref{eq10}) with (\ref{eq11}) and (\ref{eq12}), we conclude that
\[\lim_{k\to\infty}\limsup_{n\to\infty}\E[(\bar{W}_{n,A}-\E(\bar{W}_{n,A}|\mathscr{F}_k))]^2=0.\qedhere\]
\end{proof}
We are now able to complete the proof of Theorem \ref{prop1}.
\begin{proof}[Proof of Theorem \ref{prop1}]
By H\"older inequality, Lemma \ref{lem3} implies that
\[\lim_{k\to\infty}\limsup_{n\to\infty}\E(|\bar{W}_{n,A}(\theta)-\E(\bar{W}_{n,A}(\theta)|\mathscr{F}_k)|)=0.\]
By triangular inequality, we know that
\[
\begin{split}
&\E(|\bar{W}_n(\theta)-W(\theta)\E(f(N))|)\\
\leq&\E(|\bar{W}_n(\theta)-\bar{W}_{n,A}(\theta)|)+\E(|\bar{W}_{n,A}(\theta)-\E(\bar{W}_{n,A}(\theta)|\mathscr{F}_k)|)\\
&+\E\left(\left|\E(\bar{W}_{n,A}(\theta)|\mathscr{F}_k)-\E(\bar{W}_n(\theta)|\mathscr{F}_k)\right|\right)+\E(|\E(\bar{W}_n(\theta)|\mathscr{F}_k)-W(\theta)\E(f(N))|),
\end{split}
\]
which tells us that
\[
\begin{split}
&\limsup_{n\to\infty}\E(|\bar{W}_n(\theta)-W(\theta)\E(f(N))|)\\
\leq&2\limsup_{n\to\infty}\E(|\bar{W}_n(\theta)-\bar{W}_{n,A}(
\theta)|)+\limsup_{n\to\infty}\E(|\bar{W}_{n,A}(\theta)-\E(\bar{W}_{n,A}(\theta)|\mathscr{F}_k)|)\\
&+\limsup_{n\to\infty}\E\left(\left|\E(\bar{W}_n(\theta)|\mathscr{F}_k)-W(\theta)\E(f(N))\right|\right).
\end{split}
\]
Letting $k\to\infty$, by Lemma \ref{lem1} and Lemma \ref{lem3},
\[\limsup_{n\to\infty}\E(|\bar{W}_n(\theta)-W(\theta)\E(f(N))|)\leq 2\limsup_{n\to\infty}\E(|\bar{W}_n(\theta)-\bar{W}_{n,A}(\theta)|).\]
Finally, letting $A\to\infty$, by Lemma \ref{lem2}, we complete the proof.
\end{proof}
\subsection{Extreme values of a branching random walk}\label{sD}
We recall that $(V(u):u\in\T)$ is a branching random walk with reproduction law $\mathcal{L}$. For convenience and to avoid degenerate cases, we always assume that $\bP(\mathcal{L}(\R)>0)=1$ and $\bP(\mathcal{L}(\R)=1)<1$. We assume that there exists $\theta^*>0$ such that assumptions (\ref{as6}) and (\ref{as7}) hold, we define
\[Z_n:=\sum_{|u|=n}(\kappa'(\theta^*)n-V(u))e^{\theta^* V(u)-n\kappa(\theta^*)}.\]
 Recall that $\F_n=\sigma(V(u):|u|\leq n)$. It is not hard to check that $(Z_n)_{n\geq 0}$ is a martingale with respect to $\F_n$, usually called derivative martingale. Biggins and Kyprianou\cite{bigginsk} proved that $Z_n$ converges almost surely to a finite nonnegative limit, denoted by $Z$. And A\"id\'ekon\cite{aidekon2} and Chen\cite{chen} proved that $\bP(Z>0)=1$ if and only if assumption (\ref{as8}) holds. Moreover, under additional assumption (\ref{as4}), A\"id\'ekon\cite{aidekon2} showed that there exists $C^*>0$ such that
\begin{equation}\label{ai2}
\lim_{n\to\infty}\bP(M_n-m_n\leq y)=\E(e^{-C^*Ze^{-\theta^* y}})=:w(y),\quad\text{for any }y\in\R,
\end{equation}
where $m_n:=\kappa'(\theta^*)n-\frac{3}{2\theta^*}\log n$.

We define the extremal process $\mathcal{E}_n:=\sum_{|u|=n}\delta_{V(u)-m_n}$. Madaule\cite{madaule} proved that $\mathcal{E}_n$ converges in law to some randomly shifted decorated Poisson point process.
\begin{theorem}[Theorem 1.1 in \cite{madaule}]\label{th3}
If there exists $\theta^*>0$ such that assumptions (\ref{as4}), (\ref{as6}), (\ref{as7}) and (\ref{as8}) hold, then $\mathcal{E}_n$ converges in law to a randomly shifted decorated Poisson point process $\mathcal{E}$. More precisely, there exists a well-defined point process $\mathcal{D}$ on $(-\infty,0]$ such that
\[\mathcal{E}:=\sum_{i,j\geq 1}\delta_{p_i+d_j^{(i)}},\]
where $p_i$ are the atoms of a Poisson point process with random intensity $C^*\theta^*Ze^{-\theta^* x}\mathrm{d}x$, and $d_j^{(i)}$ are atoms of $\mathcal{D}^{(i)}$ that are independent copies of $\mathcal{D}$ and independent of $p_i$, where $C^*$ is defined in (\ref{ai2}).
\end{theorem}
Randomly shifted decorated Poisson point processes were further studied by Subag and Zeitouni\cite{subag}. Let $\max\mathcal{E}:=\max_{\ell\in\mathcal{E}} \ell.$ By \cite{madaule}, we know that $\max\mathcal{D}=0$ almost surely. Therefore, we know that for any $y\in\R$, $\bP(\mathcal{E}((y,\infty))<\infty)=1$. Combined with Lemma 4.4 in \cite{berestycki} and (\ref{ai2}), we obtain the following lemma.
\begin{lemma}\label{localiness}
Under the assumptions of Theorem \ref{th3}, we have $(\mathcal{E}_n,M_n-m_n)$ converges in law to $(\mathcal{E},\max\mathcal{E})$.
\end{lemma}
Finally, we introduce a result about uniform convergence of Laplace functional of point processes. Let $C_\uparrow(\R)$ be the collection of continuous and non-decreasing functions $\phi$ whose support has a finite left bound and such that for some $a\in\R$, $\phi(x)$ is positive constant for $x>a$.
\begin{lemma}\label{lemslow2}
Given point processes $(\mathcal{P}_n:n\geq 1)$ and $\mathcal{P}$ on $\R$ such that $(\max\mathcal{P}_n,\mathcal{P}_n)$ converges in law to $(\max\mathcal{P},\mathcal{P})$, $\lim_{n\to\infty}\mathcal{P}_n(\R)=\mathcal{P}(\R)$ a.s. and $\mathcal{P}_n((0,\infty))<\infty$ a.s for $n\in\N\cup\{\infty\}$. Then for any $\phi\in C_\uparrow(\R)$, we have
\begin{equation}\label{eql}
\lim_{n\to\infty}\sup_{y\in\R}|\E(e^{-\int_{\R}\phi(x+y)\mathcal{P}_n(\mathrm{d}x)})-\E(e^{-\int_{\R}\phi(x+y)\mathcal{P}(\mathrm{d}x)})|=0.
\end{equation}
\end{lemma}
\begin{proof}
Let $F_n(\phi,y):=\E(e^{-\int_{\R}\phi(x+y)\mathcal{P}_n(\mathrm{d}x)})$ and $F(\phi,y):=\E(e^{-\int_{\R}\phi(x+y)\mathcal{P}(\mathrm{d}x)})$. According to Lemma \ref{localiness}, with Lemma 4.4 in \cite{berestycki}, we conclude that for any $y\in\R$,
\[\lim_{n\to\infty} F_n(\phi,y)=F(\phi,y).\]
 Because $\phi$ is non-decreasing,  we know that $F_n(\phi,y)$ and $F(\phi,y)$ are non-increasing functions with respect to $y$. Note that $\mathcal{P}((0,\infty))<\infty$ a.s., we have $F(\phi,y)$ is continuous with respect to $y$ by dominated convergence theorem. Note that $\phi(\infty):=\lim_{x\to\infty}\phi(x)>0$ , by monotone convergence theorem and dominated convergence theorem, for any $n\in\N\cup\{\infty\}$,
\[\lim_{y\to\infty}F_n(\phi,y)=\E(e^{-\phi(\infty)\mathcal{P}_n(\R)}),\quad \lim_{y\to-\infty}F_n(\phi,y)=1.\]
Because $\lim_{n\to\infty}\mathcal{P}_n(\R)=\mathcal{P}(\R)$ a.s., by dominated convergence theorem, we obtain that
\[\lim_{n\to\infty}\lim_{y\to\infty}F_n(\phi,y)=\E(e^{-\phi(\infty)\mathcal{P}(\R)}).\]
Combined with Polya's theorem, we conclude that
\[\lim_{n\to\infty}\sup_{y\in\R}|F_n(\phi,y)-F(\phi,y)|=0.\qedhere\]
\end{proof}

\section{Proof of Theorem \ref{th1}}\label{pf1}
 Fix $t\in(0,1)$, we consider the two-speed branching random walk $(V^{(n)}(u):|u|\leq n)_{n\geq 0}$ with reproduction law $\mathcal{L}_1$ and $\mathcal{L}_2$ in the fast regime. For simplicity, we write $t_n$ for $\lfloor tn\rfloor$. Recall that $M_n^{(n)}=\max_{|u|=n}V^{(n)}(u)$, $m_n=\frac{\kappa_1(\theta)}{\theta}t_n+\frac{\kappa_2(\theta)}{\theta}(n-t_n)-\frac{1}{2\theta}\log n$ and $\mathcal{E}_n=\sum_{|u|=n}\delta_{V^{(n)}(u)-m_n}$. In this section, we aim to prove the following result, from which Theorem \ref{th1} is an immediate consequence.
\begin{proposition}\label{prop2}
Under the assumptions of Theorem \ref{th1}, for any $\phi\in C_\uparrow(\R)$, we have
\[\lim_{n\to\infty}\E(e^{-\mathcal{E}_n(\phi)})=\E(e^{-\mathcal{E}(\phi)}),\]
where $\mathcal{E}$ is the randomly shifted decorated Poisson point process defined in (\ref{limp}) and (\ref{eqex}).
\end{proposition}
The following lemma implies that with a high probability, if one particle is above $m_n+y$ at time $n$, then its ancestor at time $t_n$ must be located in a $\sqrt{t_n}$-neighborhood of $\kappa_1'(\theta) t_n$. In the rest of paper, for $i\in\{1,2\}$, we always write $(V_i(u):u\in\T)$ for a branching random walk with reproduction law $\mathcal{L}_i$.
\begin{lemma}\label{lemfast}
Under the assumptions of Theorem \ref{th1}, for any $\vep>0$ and $y\in\R$, there exists $B>0$ such that
\[\limsup_{n\to\infty}\bP(\exists |u|=n: V^{(n)}(u)>m_n+y \text{ and }V^{(n)}(u_{t_n})\notin[\kappa_1'(\theta)t_n-B\sqrt{t_n},\kappa_1'(\theta)t_n+B\sqrt{t_n}])<\vep.\]
\end{lemma}
\begin{proof}
By Markov inequality and Markov property, we have
\begin{equation}\label{eq4.0}
\begin{split}
&\bP(\exists |u|=n: V^{(n)}(u)>m_n+y \text{ and }V^{(n)}(u_{t_n})\notin[\kappa_1'(\theta)t_n-B\sqrt{t_n},\kappa_1'(\theta)t_n+B\sqrt{t_n}])\\
\leq &\E\left(\sum_{|u|=n}1_{\{V^{(n)}(u)>m_n+y,V^{(n)}(u_{t_n})\notin[\kappa_1'(\theta)t_n-B\sqrt{t_n},\kappa_1'(\theta)t_n+B\sqrt{t_n}]\}}\right)\\
=&\E\left(\sum_{|u|=t_n}\E_{V_1(u)}\left(\sum_{|v|=n-t_n}1_{\{V_2(v)>m_n+y\}}\right)1_{\{V_1(u)\notin[\kappa_1'(\theta)t_n-B\sqrt{t_n},\kappa_1'(\theta)t_n+B\sqrt{t_n}]\}}\right).\\
\end{split}
\end{equation}
 Similarly to the definition about random walk in Lemma \ref{many}, we denote by $(S^{(1)}_n)_{n\geq 0}$ a random walk associated with $(V_1(u):u\in\T)$ and $(S_n^{(2)})_{n\geq 0}$ a random walk associated with $(V_2(u):u\in\T)$. Using many-to-one formula twice, we know that the right hand side of (\ref{eq4.0}) is equal to
\begin{equation}\label{eq4.3}
\begin{split}
&\E(1_{\{S^{(1)}_{t_n}\notin[\kappa_1'(\theta)t_n-B\sqrt{t_n},\kappa_1'(\theta)t_n+B\sqrt{t_n}]\}}\E_{S^{(1)}_{t_n}}(e^{-\theta S^{(2)}_{n-t_n}+t_n\kappa_1(\theta)+(n-t_n)\kappa_2(\theta)}1_{\{S^{(2)}_{n-t_n}>m_n+y\}}))\\
= &\sqrt{n}e^{-\theta y}\E(1_{\{S^{(1)}_{t_n}\notin[\kappa_1'(\theta)t_n-B\sqrt{t_n},\kappa_1'(\theta)t_n+B\sqrt{t_n}]\}}\E_{S^{(1)}_{t_n}}(e^{-\theta( S^{(2)}_{n-t_n}-m_n-y)}1_{\{S^{(2)}_{n-t_n}>m_n+y\}}))\
\end{split}
\end{equation}
where the last equality follows from the definition of $m_n$.
By Corollary 1 in \cite{stone}, there exists a constant $C>0$ such that
\[\sup_{n\geq 1,x\in\R}\sqrt{n}\E(e^{-\theta(S^{(2)}_n-x)}1_{\{S^{(2)}_n>x\}})<C,\]
which implies that
\begin{equation}\label{eq4.2}
\begin{split}
&\sqrt{n}e^{-\theta y}\E(1_{\{S^{(1)}_{t_n}\notin[\kappa_1'(\theta)t_n-B\sqrt{t_n},\kappa_1'(\theta)t_n+B\sqrt{t_n}]\}}\E_{S^{(1)}_{t_n}}(e^{-\theta( S^{(2)}_{n-t_n}-m_n-y)}1_{\{S^{(2)}_{n-t_n}>m_n+y\}}))\\
\leq&C\bP(S^{(1)}_{t_n}\notin[\kappa_1'(\theta)t_n-B\sqrt{t_n},\kappa_1'(\theta)t_n+B\sqrt{t_n}]).
\end{split}
\end{equation}
Combining (\ref{eq4.0}) with (\ref{eq4.3}) and (\ref{eq4.2}), we conclude that
\[
\begin{split}
&\bP(\exists |u|=n: V^{(n)}(u)>m_n+y \text{ and }V^{(n)}(u_{t_n})\notin[\kappa_1'(\theta)t_n-B\sqrt{t_n},\kappa_1'(\theta)t_n+B\sqrt{t_n}])\\
\leq &C\bP(S^{(1)}_{t_n}\notin[\kappa_1'(\theta)t_n-B\sqrt{t_n},\kappa_1'(\theta)t_n+B\sqrt{t_n}]),
\end{split}
\]
which, according to central limit theorem, completes the proof by letting $n\to\infty$ and then $B\to\infty$.
\end{proof}

Now we give the proof of Proposition \ref{prop2}.

\begin{proof}[Proof of Proposition \ref{prop2}]
 For $B>0$ and $n\in\N$, we set $G(n,B):=\{|u|=t_n:|V^{(n)}(u)-\kappa_1'(\theta) t_n|\leq B\sqrt{t_n}\}$ and define a point process
\[\mathcal{G}_n:=\sum_{u\in G(n,B)}\sum_{|v|=n,v\succ u}\delta_{V^{(n)}(v)-m_n}.\]
Note that
\[
\begin{split}
0&\leq\E(e^{-\mathcal{G}_n(\phi)})-\E(e^{-\mathcal{E}_n(\phi)})\\
 &=\E[e^{-\mathcal{G}_n(\phi)}(1-e^{-\sum_{|u|=t_n, u\notin G(n,B)}\sum_{|v|=n,v\succ u}\phi(V^{(n)}(v)-m_n)})].
\end{split}
\]
Let $m_\phi$ be the infimum of the support of $\phi$, i.e. $m_\phi:=\inf\{x\in\R:\phi(x)\neq 0\}$. Thus, we know that $\phi(x)=0$ for any $x<m_\phi$. Therefore,
\begin{equation}\label{eq4.17}
\begin{split}
&|\E(e^{-\mathcal{E}_n(\phi)})-\E(e^{-\mathcal{G}_n(\phi)})|\\
\leq&\bP(\exists |u|=n: V^{(n)}(u)\geq m_n+m_{\phi} \text{ and }V^{(n)}(u_{t_n})\notin[\kappa_1'(\theta)t_n-B\sqrt{t_n},\kappa_1'(\theta)t_n+B\sqrt{t_n}]),
\end{split}
\end{equation}
the right hand side of which decays to 0 as $n\to\infty$ first and then $B\to\infty$ by Lemma \ref{lemfast}.
So we only consider how to compute the expectation $\E(e^{-\mathcal{G}_n(\phi)}).$

 By Markov property at time $t_n$, we have
\[
\begin{split}
\E(e^{-\mathcal{G}_n(\phi)})&=\E\left(\prod_{u\in\mathcal{G}(n,B)}\E_{V^{(n)}(u)}(e^{-\sum_{|v|=n-t_n}\phi(V_2(v)-m_n)})\right)\\
&=\E\left(\prod_{|u|=t_n}\E_{V_1(u)}(e^{-\sum_{|v|=n-t_n}\phi(V^{(2)}(v)-m_n)})1_{\{|V_1(u)-\kappa'_1(\theta)t_n|\leq B\sqrt{t_n}\}}\right).
\end{split}
\]
Define $M_{2,n}:=\max_{|u|=n}V_2(u)$ and $U(n,x):=\{M_{2,n}\geq m_n-x+m_\phi\}$. For fixed $x\in\R$,
\[
\begin{split}
&\E_x(e^{-\sum_{|v|=n-t_n}\phi(V_2(v)-m_n)})\\
=&1-\E(1-e^{-\sum_{|v|=n-t_n}\phi(V_2(v)-m_n+x)})\\
=&1-\E((1-e^{-\sum_{|v|=n-t_n}\phi(V_2(v)-m_n+x)})1_{U(n,x)})\\
=&1-\bP(U(n,x))(1-\E(e^{-\sum_{|v|=n-t_n}\phi(V_2(v)-m_n+x)}|U(n,x))),
\end{split}
\]
where the second equality follows from that on the event $U(n,x)$,
\[1-e^{-\sum_{|v|=n-t_n}\phi(V_2(v)-m_n+x)}=0.\]
For convenience, we write $p_{n-t_n}(x)$ for
\[\bP(U(n,x))(1-\E(e^{-\sum_{|v|=n-t_n}\phi(V_2(v)-m_n+x)}|U(n,x))).\]
Then we have
\begin{equation}\label{eq4.14}
\begin{split}
\E(e^{-\mathcal{G}_n(\phi)})=\E\left(e^{-\sum_{|u|=t_n}\frac{\log (1-p_{n-t_n}(V_1(u)))}{-p_{n-t_n}(V_1(u))}p_{n-t_n}(V_1(u))1_{\{|V_1(u)-\kappa_1'(\theta)t_n|\leq B\sqrt{t_n}\}}}\right).
\end{split}
\end{equation}
 According to Remark \ref{another}, we know that as $n\to\infty$,
\[m_n-\kappa_1'(\theta)t_n-\kappa_2'(\theta)(n-t_n)+\frac{1}{2\theta}\log n=O(1).\]
By assumption (\ref{as2}), we know that $\theta\kappa_2'(\theta)>\kappa_2(\theta)$. According to Theorem 1.3 in \cite{luo}, there exists $C_2\in(0,1)$ such that
\begin{align}
\bP(U(n,x))\nonumber=&(1+o(1))\frac{C_2}{\sqrt{n}} f\left(\frac{x-t_n\kappa_1'(\theta)}{\sqrt{t_n}}\right)e^{-\theta(m_\phi+m_n-\kappa_2'(\theta)(n-t_n)+\gamma_n-x)}e^{-(n-t_n)(\theta\kappa_2'(\theta)-\kappa_2(\theta))}\nonumber\\
=&(1+o(1))C_2f\left(\frac{x-t_n\kappa_1'(\theta)}{\sqrt{t_n}}\right)e^{\theta x-t_n\kappa_1(\theta)-\theta m_\phi}\nonumber\\
=&:(1+o(1))g_n(x)\label{lt}
\end{align}
uniformly in $|x-\kappa'_1(\theta)t_n|\leq B\sqrt{t_n}$, where $f(x):=\frac{1}{\theta\sqrt{2\pi\kappa''_2(\theta)(1-t)}}e^{-\frac{t}{1-t}\frac{x^2}{2\kappa''_2(\theta)}},$ $x\in\R$.
Combined with Theorem 1.6 in \cite{luo},we denote by $\mathcal{D}_2$  the limit in law of $\sum_{|u|=n}\delta_{V_2(u)-M_{2,n}}$ conditionally on $\{M_{2,n}\geq \kappa_2'(\theta)n\}$ and $\mathbf{e}$  an exponential variable with index $\theta$, independent of $\mathcal{D}_2$. Because $\mathcal{D}_2(\{0\})\geq 1$ a.s., we have $\E(e^{-\sum_{\ell\in \mathcal{D}_2}\phi(\ell+\mathbf{e}+m_\phi)})<1$. By (\ref{lt}) and forthcoming Lemma \ref{lemfast1}, we conclude that as $n\to\infty$,
\begin{equation}\label{eq4.15}
\sup_{|x-\kappa_1'(\theta)t_n|\leq B\sqrt{t_n}}\left|g_n(x)(1-\E(e^{-\sum_{\ell\in \mathcal{D}_2}\phi(\ell+\mathbf{e}+m_\phi)}))\frac{1}{p_{n-t_n}(x)}-1\right|\to 0.
\end{equation}
Therefore, we conclude that
\begin{equation}\label{eq4.16}
\begin{split}
&g_n(x)(1-\E(e^{-\sum_{\ell\in \mathcal{D}_2}\phi(\ell+\mathbf{e}+m_\phi)}))\\
=&C_2e^{\theta x-t_n\kappa_1(\theta)}f\left(\frac{x-\kappa'_1(\theta) t_n}{\sqrt{t_n}}\right)\int_{0}^\infty \theta e^{-\theta (y+m_\phi)}(1-\E(e^{-\sum_{\ell\in \mathcal{D}_2}\phi(\ell+y+m_\phi)}))dy \\
=&C_2e^{\theta x-t_n\kappa_1(\theta)}f\left(\frac{x-\kappa'_1(\theta) t_n}{\sqrt{t_n}}\right)\int_\R \theta e^{-\theta y}(1-\E(e^{-\sum_{\ell\in \mathcal{D}_2}\phi(\ell+y)}))dy
\end{split}
\end{equation}
where the last equality follows form that $\mathcal{D}_2((0,\infty))=0$ almost surely and $\phi(y)=0$ for any $y<m_\phi$.

Next we are going to give a lower bound and upper bound of
\[\sum_{|u|=t_n}e^{\theta V_1(u)-t_n\kappa_1(\theta)}f_B\left(\frac{V_1(u)-\kappa_1'(\theta)t_n}{\sqrt{t_n}}\right),\]
where for $x\in\R$, $f_B(x):=f(x)1_{\{|x|\leq B\}}$. For $B>0$, we define $\underline{f}_B(x):=f(x)\min\{(B-|x|)_+,1\}$, $x\in\R$.
By Theorem \ref{prop1}, we know that
\[\sum_{|u|=t_n}e^{\theta V_1(u)-t_n\kappa_1(\theta)}f\left(\frac{V_1(u)-\kappa_1'(\theta)t_n}{\sqrt{t_n}}\right)\to W^{(1)}(\theta)\E(f(N))\]
and
\[\sum_{|u|=t_n}e^{\theta V_1(u)-t_n\kappa_1(\theta)}\underline{f}_{B}\left(\frac{V_1(u)-\kappa_1'(\theta)t_n}{\sqrt{t_n}}\right)\to W^{(1)}(\theta)\E(\underline{f}_{B}(N))\]
in probability, where $W^{(1)}(\theta)$ is the limit of the additive martingale of $(V_1(u):u\in\T)$ and $N\sim\mathcal{N}(0,\kappa_1''(\theta))$. Note that for $B>0$ and $x\in\R$, we have $\underline {f}_B(x)\leq f_B(x)\leq f(x)$.
Combined with (\ref{eq4.14}), (\ref{eq4.15}) and (\ref{eq4.16}),  by dominated convergence theorem, we conclude that
\begin{align}
\E\left[e^{-C_2W^{(1)}(\theta)\E(f(N))\int_\R \theta e^{-\theta y}(1-\E(e^{-\sum_{\ell\in \mathcal{D}_2}\phi(\ell+y)}))dy}\right]\leq\liminf_{n\to\infty}\E(e^{-\mathcal{G}_n(\phi)}),\nonumber\\
 \limsup_{n\to\infty}\E(e^{-\mathcal{G}_n(\phi)})\leq \E\left[e^{-C_2W^{(1)}(\theta)\E(\underline{f}_{B}(N))\int_\R \theta e^{-\theta y}(1-\E(e^{-\sum_{\ell\in \mathcal{D}_2}\phi(\ell+y)}))dy}\right].\nonumber
\end{align}
Note that for any $x\in\R$, $\lim_{B\to\infty} \underline{f}_B(x)=f(x)$. Letting $B\to\infty$, by dominated convergence theorem, we have
\begin{equation}\label{lime}
\begin{split}
\lim_{n\to\infty}\E(e^{-\mathcal{E}_n(\phi)})=\E\left(e^{-C_{\mathrm{f}}W^{(1)}(\theta)\int_{\R}(1-\E(e^{-\sum_{\ell\in \mathcal{D}_2}\phi(\ell+y)}))\theta e^{-\theta y}\mathrm{d}y}\right),
\end{split}
\end{equation}
where
\begin{equation}\label{fast}
C_{\mathrm{f}}:=\frac{ C_2}{\theta\sqrt{2\pi}\sqrt{t\kappa_1''(\theta)+(1-t)\kappa_2''(\theta)}}.
\end{equation}

Recall that Laplace functional of a Poisson point process $\mathcal{P}$ with intensity measure $\mu$,
 \begin{equation}\label{lapf}
\E(e^{-\mathcal{P}(\phi)})=e^{-\int_{\R}(1-e^{-\phi(x)})\mu(\mathrm{d}x)}.
\end{equation}
We thus recognise the Laplace functional of the point process defined in (\ref{limp}) and (\ref{eqex}).
\end{proof}

Now we prove Theorem \ref{th1}.
\begin{proof}[Proof of Theorem \ref{th1}]
By the definition of $\mathcal{E}$ and $\max\mathcal{D}_2=0$ almost surely, for any $y\in\R$, $\bP(\mathcal{E}((y,\infty))<\infty)=1$. According to Proposition \ref{prop2}, combined with Lemma 4.4 in \cite{berestycki}, we conclude that $(\mathcal{E}_n,M_n^{(n)}-m_n)$ converges in law to $(\mathcal{E},\max \mathcal{E})$. Because $\max \mathcal{D}_2=0$ almost surely, it is not hard to see that $\max \mathcal{E}$ is distributed as a Gumbel variable with a random shift. More precisely, for any $y\in\R$, we have
\[
\bP(\max \mathcal{E}\leq y)=\E\left(e^{-C_{\mathrm{f}}W^{(1)}(\theta)e^{-\theta y}}\right). \qedhere
\]
\end{proof}

Let us emphasize the following slight generalization of Proposition 3.1 in \cite{luo}.
\begin{proposition}\label{prop3}
Consider a branching random walk $(V_2(u):u\in\T)$ with reproduction law $\mathcal{L}_2$. Fix $\theta>0$ such that $\mathcal{L}_2$ satisfies assumptions (\ref{as1}), (\ref{as3}), (\ref{as4}) and $\theta\kappa_2'(\theta)>\kappa_2(\theta)$. Then for any continuous function $\phi:\R\to [0,\infty)$ whose support has a finite left bound, $B>0$ and $x\geq 0$, we have
\[|\E(e^{-\mathcal{E}_n^{(2)}(\phi)}1_{\{M_{2,n}\geq \kappa_2'(\theta)n+y+x\}}|M_{2,n}\geq \kappa_2'(\theta)n+y)-e^{-\theta x}\E(e^{-\mathcal{D}_2(\phi)})|\to 0,\]
uniformly in $|y|\leq B\sqrt{n}$, where $\mathcal{E}^{(2)}_n:=\sum_{|u|=n}\delta_{V_2(u)-M_{2,n}}$.
\end{proposition}
\begin{proof}
The proof of this proposition is a direct adaption of the proof of Proposition 3.1 in \cite{luo}.
\end{proof}

 The strategy of the proof of the following lemma follows from Lemma 4.13 in \cite{arguin}. Recall that $m_\phi=\inf \mathrm{supp}(\phi)$.
\begin{lemma}\label{lemfast1}
Under the assumptions of Proposition \ref{prop3}. For any $\phi\in C_\uparrow(\R)$ and $B>0$, we have
\begin{equation}\label{eq3.4}
\lim_{n\to\infty}\sup_{|y|\leq B\sqrt{n}}|\E^{(y)}_n(e^{-\int_{\R}\phi(x+M^{(y)}_{2,n}+m_\phi)\mathcal{E}^{(2)}_n(\mathrm{d}x)})-\E(e^{-\int_\R \phi(x+\mathbf{e}+m_\phi)\mathcal{D}_2(\mathrm{d}x)})|=0,
\end{equation}
where $M^{(y)}_{2,n}:=M_{2,n}-n\kappa_2'(\theta)+y-m_\phi$, $\bP^{(y)}_n(\cdot):=\bP(\cdot|M_{2,n}^{(y)}\geq 0)$, $\E^{(y)}_n$ is the expectation of $\bP^{(y)}$ and $\mathbf{e}$ is an exponential variable with index $\theta$, independent of $\mathcal{D}_2$.
\end{lemma}
\begin{proof}
As $\phi\in C_\uparrow(\R)$, for $y,z\in\R$ and $I\subset\R$, we define
\begin{align}
&F_n(y,z,I):=\E^{(y)}_n(e^{-\int_{\R}\phi(x+z+m_\phi)\mathcal{E}^{(2)}_n(\mathrm{d}x)}1_{\{M^{(y)}_{2,n}\in I\}}),\nonumber\\
 &F(y,z,I):=\E(e^{-\int_\R \phi(x+z+m_\phi)\mathcal{D}_2(\mathrm{d}x)}1_{\{\mathbf{e}\in I\}}).\nonumber
\end{align}
Fix $a,b\in [0,\infty)$ such that $a<b$, we have
\begin{equation}\label{ineq2}
\begin{split}
&|\E^{(y)}_n(e^{-\int_{\R}\phi(x+M^{(y)}_{2,n}+m_\phi)\mathcal{E}^{(2)}_n(\mathrm{d}x)})-\E(e^{-\int_\R \phi(x+\mathbf{e}+m_\phi)\mathcal{D}_2(\mathrm{d}x)})|\\
\leq &\bP^{(y)}_n(M^{(y)}_{2,n}\notin (a,b])+\bP(\mathbf{e}\notin (a,b])\\
&+|\E^{(y)}_n(e^{-\int_{\R}\phi(x+M^{(y)}_{2,n}+m_\phi)\mathcal{E}^{(2)}_n(\mathrm{d}x)}1_{\{M^{(y)}_{2,n}\in(a,b]\}})-\E(e^{-\int_\R \phi(x+\mathbf{e}+m_\phi)\mathcal{D}_2(\mathrm{d}x)}1_{\{\mathbf{e}\in (a,b]\}})|.
\end{split}
\end{equation}

 Note that $\phi\in C_\uparrow(\R)$, then $\phi$ is uniformly continuous on $\R$. Therefore, fix $\vep>0$, there exists $\delta>0$ such that for any $|x_1-x_2|<\delta$, we have $|\phi(x_1)-\phi(x_2)|<\vep$. There exists a partition $\{a_1,\ldots,a_N\}$ of $(a,b]$ such that
\[a=a_0<a_1<\ldots<a_N=b\quad \text{and}\quad \max_{0\leq i\leq N-1}|a_{i+1}-a_i|<\vep.\]
Note that $1_{\{x\in(a,b]\}}=\sum_{i=0}^{N-1}1_{\{x\in I_i\}},$ where $I_i:=(a_i,a_{i+1}]$, according to $\phi$ is non-decreasing, we have the last term of right hand side of inequality (\ref{ineq2}) is smaller than
\begin{equation}\label{term}
\begin{split}
&\sum_{i=0}^{N-1}|F_n(y,a_i,I_i)-F_n(y,a_{i+1},I_i)|\\
+&\sum_{i=0}^{N-1}\{|F(y,a_i,I_i)-F(y,a_{i+1},I_i)|+|F_n(y,a_i,I_i)-F(y,a_i,I_i)|\}.
\end{split}
\end{equation}
 By triangular inequality, the first summation in (\ref{term}) is smaller than
\[
\begin{split}
&\sum_{i=0}^{N-1}\{|F_n(y,a_i,I_i)-F(y,a_i,I_i)|+|F_n(y,a_{i+1},I_i)-F(y,a_{i+1},I_i)|\}\\
+&\sum_{i=0}^{N-1}|F(y,a_i,I_i)-F(y,a_{i+1},I_i)|.
\end{split}
\]
Combined with (\ref{ineq2}), according to Proposition \ref{prop3}, we conclude that
\begin{align}
&\limsup_{n\to\infty}\sup_{|y|\leq B\sqrt{n}}|\E^{(y)}_n(e^{-\int_{\R}\phi(x+M^{(y)}_{2,n}+m_\phi)\mathcal{E}^{(2)}_n(\mathrm{d}x)})-\E(e^{-\int_\R \phi(x+\mathbf{e}+m_\phi)\mathcal{D}_2(\mathrm{d}x)})|\label{ineq3}\\
\leq &2\bP(\mathbf{e}\notin(a,b])+\limsup_{n\to\infty}\sup_{|y|\leq B\sqrt{n}}2\sum_{i=0}^{N-1}|F(y,a_i,I_i)-F(y,a_{i+1},I_i)|\nonumber.
\end{align}
For $k\in\N$, we have
\[
\begin{split}
&\sum_{i=0}^{N-1}|F(y,a_i,I_i)-F(y,a_{i+1},I_i)|\\
\leq &\bP(\mathcal{D}_2([-b,\infty))\geq k)\\
&\quad+\sum_{i=0}^{N-1}\E\left(\int_{\R}|\phi(x+a_i+m_\phi)-\phi(x+a_{i+1}+m_\phi)|\mathcal{D}_2(\mathrm{d}x)1_{\{\mathbf{e}\in I_i,\mathcal{D}_2([-b,\infty))<k\}}\right)\\
\leq & \bP(\mathcal{D}_2([-b,\infty))\geq k)+k\vep,
\end{split}
\]
where the second inequality follows from a basic inequality that is $|e^{-x_1}-e^{-x_2}|\leq |x_1-x_2|$ for $x_1,x_2>0$.
 Combined with (\ref{ineq3}), we deduce that
\[
\begin{split}
&\limsup_{n\to\infty}\sup_{|y|\leq B\sqrt{n}}|\E^{(y)}_n(e^{-\int_{\R}\phi(x+M^{(y)}_{2,n}+m_\phi)\mathcal{E}^{(2)}_n(\mathrm{d}x)})-\E(e^{-\int_\R \phi(x+\mathbf{e}+m_\phi)\mathcal{D}_2(\mathrm{d}x)})|\\
\leq& 2\bP(\mathbf{e}\notin(a,b])+2\bP(\mathcal{D}_2([-b,\infty))\geq k)+2k\vep.
\end{split}
\]
Note that $\bP(\mathcal{D}_2((z,\infty))<\infty)=1$ for $z\in\R$. Letting $\vep \to 0$ first, then $k\to\infty$, $a\to 0$ and $b\to\infty$, we complete the proof.
\end{proof}
\section{Proof of Theorem \ref{th2}}\label{proofofth2}
In this section, fix $t\in(0,1)$, we consider a two-speed branching random walk with reproduction law $\mathcal{L}_1$ and $\mathcal{L}_2$ in the slow regime. For convenience, we write $t_n$ for $\lfloor tn\rfloor$. More precisely, we assume that there exist $\theta_1^*,\theta_2^*>0$ such that $\theta_1^*<\theta_2^*$ and $\theta_i^*\kappa_i'(\theta_i^*)=\kappa_i(\theta_i^*)$, $i=1,2$.
Recall that $m_n=\kappa_1'(\theta_1^*)t_n+\kappa_2'(\theta_2^*)(n-t_n)-\frac{3}{2\theta^*_1}\log(t_n)-\frac{3}{2\theta^*_2}\log(n-t_n)$ and $\mathcal{E}_n=\sum_{|u|=n}\delta_{V^{(n)}(u)-m_n}.$ We first introduce an important proposition that shows the weak convergence of the Laplace functional of the extremal process $\mathcal{E}_n$. The proof of forthcoming proposition is a direct application of results by A\"id\'ekon\cite{aidekon2} and Madaule\cite{madaule}. Finally, we prove Theorem \ref{th2} by the Proposition.
\begin{proposition}\label{prop4}
Under the assumptions of Theorem \ref{th2}. For any $\phi\in C_\uparrow(\R)$, we have
\[\lim_{n\to\infty}\E(e^{-\mathcal{E}_n(\phi)})=\E(e^{-\mathcal{E}(\phi)}),\]
where $\mathcal{E}$ is defined in (\ref{eqth21}).
\end{proposition}
Before the proof of Proposition \ref{prop4}, we introduce an important lemma that give the probability estimate of particles' trajectories. We define $m^{(1)}_n:=\kappa_1'(\theta_1^*)n-\frac{3}{2\theta_1^*}\log n$ and $m_n^{(2)}:=\kappa_2'(\theta_2^*)n-\frac{3}{2\theta_2^*}\log n$. It is not hard to see that $m_n=m^{(1)}_{t_n}+m^{(2)}_{n-t_n}$. Recall that for $i\in\{1,2\}$, $(V_i(u):u\in\T)$ is a branching random walk with reproduction law $\mathcal{L}_i$.
\begin{lemma}\label{lemslow}
Under the assumptions of Theorem \ref{th2}, for any $\vep>0$ and $y\in\R$, there exists $B>0$ such that
\[\lim_{B\to\infty}\sup_{n\in\N}\bP(\exists |u|=n : V^{(n)}(u)>m_n+y,\ |V^{(n)}(u_{t_n})-m^{(1)}_{t_n}|\geq B)=0.\]
\end{lemma}
\begin{proof}
Note that
\begin{equation}\label{q4}
\begin{split}
&\bP(\exists |u|=n : V^{(n)}(u)>m_n+y,\ |V^{(n)}(u_{t_n})-m^{(1)}_{t_n}|\geq B)\\
\leq &\bP(\exists |u|=n : V^{(n)}(u)>m_n+y,\ V^{(n)}(u_{t_n})-m^{(1)}_{t_n}\geq B)\\
&+\bP(\exists |u|=n : V^{(n)}(u)>m_n+y,\ V^{(n)}(u_{t_n})-m^{(1)}_{t_n}\leq -B).
\end{split}
\end{equation}
It is not hard to see that there exists $C>0$ such that
\begin{equation}\label{q5}
\begin{split}
&\bP(\exists |u|=n : V^{(n)}(u)>m_n+y,\ V^{(n)}(u_{t_n})-m^{(1)}_{t_n}\geq B)\\
\leq &\bP(\exists |u|=t_n : \bar{V}(u)-m^{(1)}_{t_n}\geq B) \leq C(1+B)e^{-\theta_1^* B},
\end{split}
\end{equation}
where the second inequality follows from Theorem 4.2 in \cite{bastien}.

For $\lambda\in (0,1)$ and $1\leq k\leq t_n$, let
\[f_k^{(n)}:=\kappa_1'(\theta_1^*)k1_{\{k\leq \lambda t_n\}}+\left[\kappa_1'(\theta_1^*)k-\frac{3}{2\theta_1^*}\log \left(\frac{t_n}{t_n-k+1}\right)\right]1_{\{k> \lambda t_n\}}.\]
For $B>0$, we denote by $J^\lambda_{B}(n)$ the following collection of particles:
\[\left\{u\in\T:|u|=t_n,V^{(n)}(u_k)\leq f_k^{(n)}+B,\forall 1\leq k\leq t_n\right\},\]

Note that
\begin{equation}\label{q1}
\begin{split}
&\bP(\exists |u|=n : V^{(n)}(u)>m_n+y,\ V^{(n)}(u_{t_n})-m^{(1)}_{t_n}\leq -B)\\
\leq& \bP(\exists |u|=t_n : V^{(n)}(u)-m^{(1)}_{t_n}\leq -B, \max_{v>u,|v|=n}V^{(n)}(v)>m_n+y)\\
\leq &\bP(\exists |u|=t_n : V^{(n)}(u)-m^{(1)}_{t_n}\leq -B, \max_{v>u,|v|=n}V^{(n)}(v)>m_n+y, u\in J_B^\lambda(n))\\
&+\bP(\exists |u|=t_n: u\notin J^\lambda_{B}(n)).
\end{split}
\end{equation}
By Lemma 4.3 and Lemma 4.4 in \cite{bastien}, there exists some constant $C>0$, independent of $B,t$, such that
\begin{equation}\label{q2}
\bP(\exists |u|=t_n: u\notin J^\lambda_{B}(n))\leq C(1+ B)e^{-\theta_1^*B}.
\end{equation}

  We define
\[Z_n(B,i):=\sum_{|u|=t_n}1_{\{V_1(u)-m^{(1)}_{t_n}\in(-B-i-1,-B-i],V_1(u_k)-f_k^{(n)}\leq  B,\forall 1\leq k \leq t_n\}}.\]
 Note that
\[
\begin{split}
&\bP(\exists |u|=t_n : V^{(n)}(u)-m^{(1)}_{t_n}\leq -B, \max_{v>u,|v|=n}V^{(n)}(v)>m_n+y, u\in J^\lambda_{B}(n))\\
\leq &\sum_{i=0}^\infty\E\left(\sum_{|u|=t_n}1_{\{V^{(n)}(u)-m^{(1)}_{t_n}\in(-B-i-1,-B-i],u\in J^\lambda_{B}(n), \max_{v>u,|v|=n}V^{(n)}(v)>m_n+y\}}\right)\\
\leq &\sum_{i=0}^\infty \E(Z_n(B,i))\bP(M_{2,n-t_n}>m^{(2)}_{n-t_n}+y+B+i),
\end{split}
\]
where the second equality follows from Markov property at time $t_n$ and $M_{2,n}:=\max_{|u|=n}V_2(u)$.
By Lemma 4.5 in \cite{bastien}, there exists $C>0$ such that
\[
\E(Z_n(B,i))\leq Ce^{\theta_1^* (B+i+1)}(2B+i+2)(1+B)
\]
 On the other hand, by Theorem 4.2 in \cite{bastien}, there exists $C>0$ such that for any $B>0$ and $i\geq 0$,
\[\bP(M_{2,n-t_n}>m^{(2)}_{n-t_n}+y+B+i)\leq Ce^{-\theta_2^*(y+B+i)}.\]
Therefore, we conclude that
\begin{equation}\label{q3}
\begin{split}
&\bP(\exists |u|=tn : V^{(n)}(u)-m^{(1)}_{t_n}\leq -B, \max_{v>u,|v|=n}V^{(n)}(v)>m_n+y, u\in J^\lambda_{B}(n))\\
\leq &C(1+ B)e^{-\theta_2^* y+\theta_1^*}\sum_{i=0}^\infty (2B+i+2)e^{-(\theta_2^*-\theta_1^*)(B+i)}\\
\leq &C(1+B^2) e^{-(\theta_2^*-\theta_1^*)B}
\end{split}
\end{equation}
where the second equality follows from that $\sum_{i=0}^\infty ie^{-(\theta_2^*-\theta_1^*)i}<\infty$ because $\theta_1^*<\theta_2^*$.

With (\ref{q4})-(\ref{q3}), letting $B\to\infty$, we completes the proof.
\end{proof}
Now we start to prove Proposition \ref{prop4}.
\begin{proof}[Proof of Proposition \ref{prop4}]
Define $\mathcal{H}(n,B):=\{|u|=tn: |V^{(n)}(u)-m^{(1)}_{t_n} |<B\}$.
For any $n\in\N$, $B>0$ and $\phi\in C_\uparrow(\R)$, note that
\[
\begin{split}
&|\E(e^{-\mathcal{E}_n(\phi)})-\E(e^{-\sum_{u\in\mathcal{H}(n,B)}\sum_{v\succ u,|v|=n}\phi(V^{(n)}(v)-m_n)})|\\
\leq &\E(|1-e^{-\sum_{|u|=t_n,u\notin\mathcal{H}(n,B)}\sum_{v\succ u,|v|=n}\phi(V^{(n)}(v)-m_n)}|)\\
\leq &\bP(\exists |u|=n,\ V^{(n)}(u)-m_n> m_\phi,\ |V^{(n)}(u_{t_n})-\bar{m}_{t_n}|\geq B),
\end{split}
\]
where $m_\phi=\inf\{x\in\R:\phi(x)\neq 0\}.$
Thus, by Lemma \ref{lemslow}, we only need to compute the expectation
\[\E(e^{-\sum_{u\in\mathcal{H}(n,B)}\sum_{v\succ u,|v|=n}\phi(V^{(n)}(v)-m_n)}).\]
By Theorem \ref{th3}, we denote by $\mathcal{E}_n^{(2)}:=\sum_{|u|=n}\delta_{V_2(u)-m^{(2)}_n}$ the extremal process of $(V_2(u):u\in\T)$ and $\mathcal{E}^{(2)}$ the limit in law of $\mathcal{E}^{(2)}_n$.
By Markov property at time $t_n$, we have
\[
\begin{split}
&\E(e^{-\sum_{u\in \mathcal{H}(n,B)}\sum_{v\succ u,|v|=n}\phi(V^{(n)}(v)-m_n)})\\
=&\E\left(\prod_{|u|=t_n}\E_{V_1(u)-m^{(1)}_{t_n}}(e^{-\mathcal{E}^{(2)}_{n-t_n}(\phi)})1_{\{|V_1(u)-m^{(1)}_{t_n}|<B\}}\right)\\
=&\E\left(e^{-\sum_{|u|=t_n}[-\log\E_{V_1(u)-m^{(1)}_{t_n}}(e^{-\mathcal{E}^{(2)}_{n-t_n}(\phi)})]1_{\{|V_1(u)-m^{(1)}_{t_n}|<B\}}}\right).
\end{split}
\]
By direct computation, there exists $c_\phi\in\R$ such that for any $x\in\R$, $\E(e^{-\sum_{\ell\in\mathcal{E}^{(2)}}\phi(x+\ell)})=w^{(2)}(c_\phi-x)$, where $w^{(2)}(x):=\bP(\max \mathcal{E}^{(2)}\leq x)$. Using Lemma \ref{localiness} and Lemma \ref{lemslow2}, we know that
\[\lim_{n\to\infty}\sup_{|x|<B}\left|\frac{\E_x(e^{-\mathcal{E}^{(2)}_n(\phi)})}{\E(e^{-\sum_{\ell\in \mathcal{E}^{(2)}}\phi(x+\ell)})}-1\right|=0.\]
Therefore, we only need to consider how to estimate $\E(e^{-\mathcal{E}^{(1)}_{t_n}(g_{\phi,B})})$,
where
\[\mathcal{E}^{(1)}_{n}=\sum_{|u|=n}\delta_{V_1(u)-m^{(1)}_{n}}\quad\text{and}\quad g_{\phi,B}(x):=-\log(\E(e^{-\sum_{\ell\in \mathcal{E}^{(2)}}\phi(x+\ell)}))1_{\{|x|<B\}}.\]
We denote by $C_c(\R)$ the collection of continuous functions with compact support. Because $\inf_{|x|\leq B}\E(e^{-\sum_{\ell\in \mathcal{E}^{(2)}}\phi(x+\ell)})>0$, there exist two non-negative function sequences $(\bar{g}_{\phi,B,k})_{k\geq 1}$ and $(\underline{g}_{\phi,B,k})_{k\geq 1}$ such that for any $k\geq 1$, $\bar{g}_{\phi,B,k},\underline{g}_{\phi,B,k}\in C_c(\R)$ and for any $x\in\R$,
\[\underline{g}_{\phi,B,k}(x)\nearrow g_{\phi,B}(x)\quad\text{and}\quad \bar{g}_{\phi,B,k}(x)\searrow \bar{g}_{\phi,B}(x),\quad\text{as}\quad k\to\infty,\]
where $\bar{g}_{\phi,B}(x):=-(\log \E(e^{-\sum_{\ell\in\mathcal{E}^{(2)}}\phi(\ell+x)}))1_{\{|x|\leq B\}}$, $x\in\R$.
Let $\mathcal{E}^{(1)}$ be the limit in law of $\mathcal{E}^{(1)}_n$. According to Theorem \ref{th3} and Lemma 5.1 in \cite{kallenberg}, for any $k\geq 1$, we have
\begin{equation}\label{eq5.2}
\E(e^{-\mathcal{E}^{(1)}(\bar{g}_{\phi,B,k})})\leq \liminf_{n\to\infty}\E(e^{-\mathcal{E}^{(1)}_{t_n}(g_{\phi,B})})\leq \limsup_{n\to\infty}\E(e^{-\mathcal{E}^{(1)}_{t_n}(g_{\phi,B})})\leq \E(e^{-\mathcal{E}^{(1)}(\underline{g}_{\phi,B,k})}).
\end{equation}
By dominated convergence theorem and monotone convergence theorem, letting $k\to\infty$ first and then $B\to\infty$, we conclude that
\[\lim_{n\to\infty}\E(e^{-\mathcal{E}_n(\phi)})=\E(e^{-\mathcal{E}^{(1)}(g_\phi)}),\]
where $g_\phi(x):=-\log(\E(e^{-\sum_{\ell\in \mathcal{E}^{(2)}}\phi(x+\ell)}))$, $x\in\R$.
Let $(\mathcal{E}^{(2,i)}:i\geq 1)$ be independent copies of $\mathcal{E}^{(2)}$ and be independent of $\mathcal{E}^{(1)}$. Then it is not hard to check that
\[\E(e^{-\mathcal{E}^{(1)}(g_\phi)})=\E(e^{-\mathcal{E}(\phi)}),\]
where $\mathcal{E}:=\sum_{i,j\geq 1}\delta_{e_i+e_j^{(i)}}$ and $(e_i:i\geq 1)$ and $(e_j^{(i)}:j\geq 1)$ are the atoms of $\mathcal{E}^{(1)}$ and $\mathcal{E}^{(2,i)}$ respectively.
\end{proof}
In the following lemma, we prove that the maximum of $\mathcal{E}$ is almost surely finite.
\begin{lemma}\label{lemslow1}
Under the assumptions of Theorem \ref{th2}, we have $\bP(\max \mathcal{E}<\infty)=1$.
\end{lemma}
\begin{proof}
Recall that $w^{(2)}(x)=\bP(\max \mathcal{E}^{(2)}\leq x),$ $x\in\R$. Note that for $y\in\R$,
\begin{equation}\label{lb}
\begin{split}
\bP(\max \mathcal{E}\leq y)&=\bP(\forall e_i\in \mathcal{E}^{(1)}, e_i+\max\mathcal{E}^{(2,i)}\leq y)\\
&=\E\left(\prod_{e_i\in\mathcal{E}^{(1)}}w^{(2)}(y-e_i)\right)\\
&\geq \E(e^{-\mathcal{E}^{(1)}(-\log w_2(y-\cdot))}1_{\{\max\mathcal{E}^{(1)}\leq y\}}),
\end{split}
\end{equation}
where the second equality follows from the independence between $\mathcal{E}^{(1)}$ and $(\mathcal{E}^{(2,i)}:i\geq 1)$.
By Proposition 1.3 in \cite{aidekon2}, there exists $C>0$ such that for any $x\geq 0$,
\[1-w^{(2)}(x)\leq C(1+x)e^{-\theta_2^* x}, \]
which implies that
\[
-\log w^{(2)}(x)\leq C(1+x)e^{-\theta_2^* x}.
\]
Therefore, given $\bar{\theta}\in(\theta_1^*,\theta_2^*)$, we conclude that for $x\geq 0$,
\[
-\log w^{(2)}(x)\leq Ce^{-\bar{\theta}x},
\]
which implies that
\begin{equation}\label{lb1}
\E(e^{-\mathcal{E}^{(1)}(-\log w_2(y-\cdot))}1_{\{\max\mathcal{E}^{(1)}\leq y\}})\geq\E(e^{-Ce^{-\bar{\theta}y\mathcal{E}^{(1)}(e^{\bar{\theta}\cdot})}}1_{\{\max\mathcal{E}^{(1)}\leq y\}}).
\end{equation}
By Theorem 2.3 in \cite{madaule}, for $\theta>\theta_1^*$, $\bP(\mathcal{E}^{(1)}(e^{\theta \cdot})<\infty)=1.$ Note that $\max\mathcal{E}^{(1)}<\infty$ a.s., combined with (\ref{lb}) and (\ref{lb1}), according to dominated convergence theorem, we conclude that
\[\lim_{y\to\infty}\bP(\max\mathcal{E}\leq y)=1,\]
which completes the proof.
\end{proof}
Now we prove Theorem \ref{th2}.
\begin{proof}[Proof of Theorem \ref{th2}]
According to (\ref{lb}), for $y\in\R$
\begin{equation}\label{eq6.0}
\bP(\max\mathcal{E}\leq y)=\E(e^{-\mathcal{E}^{(1)}(w_y)}),
\end{equation}
where $w_y(x):=-\log w^{(2)}(y-x)$. Let $Z^{(1)}$ be the limit of the derivative martingale of $(V_1(u):u\in\T)$. According to Theorem \ref{th3} and (\ref{lapf}), we know that there exist a point process $\mathcal{D}^{(1)}$ on $(-\infty,0]$ and some constant $C^*_1>0$ such that
\[
\begin{split}
\E(e^{-\mathcal{E}^{(1)}(w_y)})&=\E(e^{-C^*_1\theta_1^*Z^{(1)}\int_{\R}(1-\E(e^{-\mathcal{D}^{(1)}(w_y(\cdot+s))}))e^{-\theta_1^* s}\mathrm{d}s})\\
&=\E(e^{-C_1^*\theta_1^*Z^{(1)}[\int_{\R}(1-\E(e^{-\mathcal{D}^{(1)}(w_s)}))e^{\theta_1^* s}\mathrm{d}s]e^{-\theta_1^* y}})=\E(e^{-C_{\mathrm{s}}Z^{(1)}e^{-\theta_1^* y}}),
\end{split}
\]
where the second equality follows from a change of variables $s\to y-s$ and
\begin{equation}\label{eq6.1}
C_{\mathrm{s}}:=C_1^*\theta_1^*\int_{\R}(1-\E(e^{-\mathcal{D}^{(1)}(w_s)}))e^{\theta_1^* s}\mathrm{d}s\in(0,\infty),
\end{equation}
where $C_{\mathrm{s}}<\infty$ follows from Lemma \ref{lemslow1}. By direct computation, for any $\phi\in C_\uparrow(\R)$, there exists $c_{\phi}\in\R$ such that for any $y\in\R$, $\E(e^{-\mathcal{E}^{(2)}(\phi(\cdot+y))})=w^{(2)}(c_{\phi}-y)$, which implies that for $z>0$,
\[\E(e^{-z\mathcal{E}(\phi)})=\E(e^{-\mathcal{E}^{(1)}(w_{c_{z\phi}})})=\E(e^{-C_{\mathrm{s}}Z^{(1)}e^{-\theta_1^* c_{z\phi}}}).\]
As $\lim_{z\to0+}\E(e^{-z\mathcal{E}(\phi)})=1$, we have $\lim_{z\to 0+}c_{z\phi}=\infty$, which yields that
\[\lim_{z\to 0+}\E(e^{-z\mathcal{E}(\phi)})=1,\]
which yields that for any $y\in\R$, $\bP(\mathcal{E}((y,\infty))<\infty)=1$. According to Lemma 4.4 in \cite{berestycki} and Proposition \ref{prop4}, we conclude that $(M_n^{(n)}-m_n, \mathcal{E}_n)$ converges in law to $(\max\mathcal{E},\mathcal{E})$. Moreover, by Theorem 10 in \cite{subag}, we conclude that $\mathcal{E}$ can be rewritten as a randomly shifted decorated Poisson point process with some decoration $\mathcal{D}'$ such that $\max\mathcal{D}'=0$ almost surely.
\end{proof}

\noindent\textbf{Acknowledgements:}
I wish to thank my supervisor Bastien Mallein for introducing me to this subject and constantly finding the time for useful discussions and advice. I has received financial support from China Scholarship Council (No.202306040031) and the CNRS through the MITI interdisciplinary program 80PRIME GEX-MBB.

\end{document}